\documentclass[10pt]{amsart}

\usepackage{amsmath, amsthm, amssymb}

\usepackage{multicol}
\usepackage[nopostdot, nonumberlist, style=index]{glossaries}
\usepackage{glossary-mcols}

\setglossarystyle{mcolindex}
\makenoidxglossaries

\usepackage[initials, alphabetic]{amsrefs}
\usepackage[colorlinks=true, linktoc=page, linkcolor=webgreen, citecolor=webgreen, urlcolor=webbrown]{hyperref}
\usepackage[dvipsnames]{xcolor}
\usepackage[mathscr]{eucal}

\usepackage{enumitem}
\usepackage{thmtools}
\usepackage{geometry} 
\usepackage{graphicx} 
\usepackage{float} 
\usepackage{wrapfig}
\usepackage{tikz-cd}
\usepackage[utf8]{inputenc} % Required for inputting international characters
\usepackage[T1]{fontenc} % Output font encoding for international characters
\usepackage{XCharter} % Use the XCharter fonts

\usepackage{geometry}
\usepackage{setspace}
\newtheorem{thm}{Theorem}[section]
\newtheorem{state}{Theorem}

\newtheorem{prop}[thm]{Proposition}
\newtheorem{lem}[thm]{Lemma}
\newtheorem{cor}[thm]{Corollary}

\theoremstyle{remark}
\newtheorem{dfn}[thm]{Definition}
\newtheorem{rmk}[thm]{Remark}
\newtheorem{note}[thm]{Notation}
\newtheorem{conv}[thm]{Convention}
\newtheorem{ex}[thm]{Example}
\theoremstyle{definition}
\newtheorem*{ack}{Acknowledgements}

\numberwithin{equation}{section}

\newcommand{\tc}{\textbf{t}} 
\newcommand{\qc}{\textbf{q}}	%coefficients t,q

\DeclareMathOperator{\SL}{SL}
\DeclareMathOperator{\ord}{ord}
\DeclareMathOperator{\Sym}{Sym}
\DeclareMathOperator{\pr}{pr}
\newcommand{\Rat}{\mathbb Q}
\newcommand{\Z}{\mathbb Z}

\def\sp{$^+$stable }
\def\sm{$^-$stable }
\def\mfa{\mathfrak{a}} % automorphism
\def\mfb{\mathfrak{b}} % automorphism
\def\mfe{\mathfrak{e}} %anti-automorphisms
\def\mfi{\iota} % automorphism
\def\ct{\mathfrak{c}} % automorphism
\def\A{\mathscr{A}}
\def\Atq{\mathbb{A}_{\tc,\qc}}
\def\Aqt{\mathbb{A}_{\qc,\tc}}

\def\Bo{\mathscr{B}}
\def\T{\mathscr{T}}
\def\X{\mathscr{X}}
\def\Xb{\mathbf{X}}

\def\Y{\mathscr{Y}}
\def\P{\mathscr{P}}
\def\Pas{\mathscr{P}_{\rm as}}
\def\H{\mathscr{H}}
\def\AHA{\mathscr{A}}
\def\Ht{\widetilde{\mathscr{H}}}

\definecolor{webgreen}{rgb}{0,.4,0}
\definecolor{webbrown}{rgb}{.4,0,0}
\newglossaryentry{ct}
{
  name={\ensuremath{\ct_{y_k} }},
  description={\hfill\S\ref{sec: ddparep}},
  sort=ct
}

\newglossaryentry{Xb}
{
  name={\ensuremath{\Xb, \Xb_k, }},
  description={\hfill\S\ref{sec: not1}},
  sort=x
}

\newglossaryentry{Xbar}
{
  name={\ensuremath{\overline{\Xb}_k, \overline{\Xb}_{[k,m]} }},
  description={\hfill\S\ref{sec: not1}},
  sort=xb
}

\newglossaryentry{pn}
{
  name={\ensuremath{p_n[\Xb]}},
  description={\hfill\S\ref{sec: not1}},
  sort=p
}

\newglossaryentry{en}
{
  name={\ensuremath{e_n[\Xb] }},
  description={\hfill\S\ref{sec: not1}},
  sort=en
}

\newglossaryentry{m}
{
  name={\ensuremath{m_\lambda[\Xb] }},
  description={\hfill\S\ref{sec: not1}},
  sort=m
}

\newglossaryentry{hn}
{
  name={\ensuremath{h_n[\Xb]}},
  description={\hfill\S\ref{sec: not1}},
  sort=hna
}

\newglossaryentry{hnf}
{
  name={\ensuremath{h_n[(1-\tc)\overline{\Xb}_k]}},
  description={\hfill\S\ref{sec: esf}},
  sort=hnb
}

\newglossaryentry{Exp}
{
  name={\ensuremath{{\rm Exp}[\Xb] }},
  description={\hfill\S\ref{sec: not2}},
  sort=exp
}

\newglossaryentry{Pn}
{
  name={\ensuremath{\P_k}},
  description={\hfill\S\ref{sec: not4}},
  sort=pa
}

\newglossaryentry{Pnpm}
{
  name={\ensuremath{\P_k^\pm }},
  description={\hfill\S\ref{sec: not4}},
  sort=pa
}

\newglossaryentry{P(n)}
{
  name={\ensuremath{\P(k)^\pm }},
  description={\hfill\S\ref{sec: p(k)}},
  sort=pab
}

\newglossaryentry{Pbullet}
{
  name={\ensuremath{\P_\bullet }},
  description={\hfill\S\ref{sec: ddpaP}},
  sort=pask
}

\newglossaryentry{Pas}
{
  name={\ensuremath{\Pas^\pm }},
  description={\hfill\S\ref{sec: pas}},
  sort=pas
}

\newglossaryentry{Pinf}
{
  name={\ensuremath{\P_\infty^\pm }},
  description={\hfill\S\ref{sec: maps}},
  sort=paa
}

\newglossaryentry{sn}
{
  name={\ensuremath{s_i}},
  description={\hfill\S\ref{sec: not4}},
  sort=s
}

\newglossaryentry{pik}
{
  name={\ensuremath{\pi_k }},
  description={\hfill\S\ref{sec: maps}},
  sort=pik
}

\newglossaryentry{Pik}
{
  name={\ensuremath{\Pi_k }},
  description={\hfill\S\ref{sec: maps}},
  sort=pik
}

\newglossaryentry{iota}
{
  name={\ensuremath{\iota }},
  description={\hfill\S\ref{sec: sl2}},
  sort=iota
}

\newglossaryentry{ik}
{
  name={\ensuremath{\iota_k }},
  description={\hfill\S\ref{sec: maps}},
  sort=iotak
}

\newglossaryentry{iotak}
{
  name={\ensuremath{\iota(k) }},
  description={\hfill\S\ref{sec: ddpaP}},
  sort=iotakk
}

\newglossaryentry{Ik}
{
  name={\ensuremath{I_k }},
  description={\hfill\S\ref{sec: maps}},
  sort=iotakb
}

\newglossaryentry{J}
{
  name={\ensuremath{J }},
  description={\hfill\S\ref{sec: maps}},
  sort=j
}

\newglossaryentry{Jk}
{
  name={\ensuremath{J_k }},
  description={\hfill\S\ref{sec: maps}},
  sort=jk
}

\newglossaryentry{Ak}
{
  name={\ensuremath{\A_k}},
  description={\hfill\S\ref{sec: AHA}},
  sort=ak
}

\newglossaryentry{tomegak}
{
  name={\ensuremath{\widetilde{\omega}_k}},
  description={\hfill\S\ref{sec: AHA},\ref{sec: standardrep}},
  sort=omega
}

\newglossaryentry{omegak}
{
  name={\ensuremath{\omega_k}},
  description={\hfill\S\ref{sec: DAHA},\ref{sec: standardrep}},
  sort=omega
}

\newglossaryentry{Hk}
{
  name={\ensuremath{\H_k}},
  description={\hfill\S\ref{sec: DAHA}},
  sort=h
}

\newglossaryentry{Hkpm}
{
  name={\ensuremath{\H_k^\pm}},
  description={\hfill\S\ref{sec: DAHA}},
  sort=h
}

\newglossaryentry{tHk}
{
  name={\ensuremath{\widetilde{\H}_k^+}},
  description={\hfill\S\ref{sec: dDAHA}},
  sort=hk
}

\newglossaryentry{H}
{
  name={\ensuremath{\H^\pm}},
  description={\hfill\S\ref{sec: sDAHA}},
  sort=h
}

\newglossaryentry{H(k)}
{
  name={\ensuremath{\H(k)^\pm}},
  description={\hfill\S\ref{sec: sDAHA}},
  sort=hkp
}

\newglossaryentry{e}
{
  name={\ensuremath{\mfe}},
  description={\hfill\S\ref{sec: sDAHA}},
  sort=e
}

\newglossaryentry{a}
{
  name={\ensuremath{\mfa}},
  description={\hfill\S\ref{sec: sl2}},
  sort=a
}

\newglossaryentry{b}
{
  name={\ensuremath{\mfb}},
  description={\hfill\S\ref{sec: sl2}},
  sort=b
}

\newglossaryentry{T}
{
  name={\ensuremath{T_i}},
  description={\hfill\S\ref{sec: DAHA}},
  sort=tnor
}

\newglossaryentry{X}
{
  name={\ensuremath{X_i}},
  description={\hfill\S\ref{sec: DAHA}},
  sort=xi
}

\newglossaryentry{Y}
{
  name={\ensuremath{Y_i}},
  description={\hfill\S\ref{sec: DAHA}},
  sort=yi
}

\newglossaryentry{Y(k)}
{
  name={\ensuremath{Y_i^{(k)}}},
  description={\hfill\S\ref{sec: standardrep}},
  sort=yi
}

\newglossaryentry{varpi}
{
  name={\ensuremath{\varpi_k}},
  description={\hfill\S\ref{sec: dDAHA},\ref{sec: dDAHArep}},
  sort=omegapi
}

\newglossaryentry{gamma}
{
  name={\ensuremath{\gamma_k}},
  description={\hfill\S\ref{sec: dDAHA},\ref{sec: dDAHArep}},
  sort=gamma
}

\newglossaryentry{tY}
{
  name={\ensuremath{\widetilde{Y}_i}},
  description={\hfill\S\ref{sec: dDAHA}},
  sort=yit
}

\newglossaryentry{tY(k)}
{
  name={\ensuremath{\widetilde{Y}_i^{(k)}}},
  description={\hfill\S\ref{sec: dDAHArep}},
  sort=yit
}

\newglossaryentry{tYinf}
{
  name={\ensuremath{\widetilde{Y}_i^{(\infty)}}},
  description={\hfill\S\ref{sec: ind}},
  sort=yiti
}

\newglossaryentry{tZ}
{
  name={\ensuremath{ \widetilde{Z}_i^{(k)}}},
  description={\hfill\S\ref{sec: W}},
  sort=zk
}

\newglossaryentry{tZinf}
{
  name={\ensuremath{\widetilde{Z}_i^{(\infty)}}},
  description={\hfill\S\ref{sec: W}},
  sort=zkinf
}

\newglossaryentry{W}
{
  name={\ensuremath{W_i}},
  description={\hfill\S\ref{sec: W}},
  sort=w
}

\newglossaryentry{W(k)}
{
  name={\ensuremath{W_i^{(k)}}},
  description={\hfill\S\ref{sec: W}},
  sort=w
}

\newglossaryentry{Winf}
{
  name={\ensuremath{W_i^{(\infty)}}},
  description={\hfill\S\ref{sec: Wlim}},
  sort=winf
}

\newglossaryentry{pr}
{
  name={\ensuremath{\pr_i}},
  description={\hfill\S\ref{sec: dDAHArep}},
  sort=pr
}

\newglossaryentry{Tcal}
{
  name={\ensuremath{\T_i}},
  description={\hfill\S\ref{sec: -limits},\ref{sec: +limits}},
  sort=tt
}

\newglossaryentry{Xcal}
{
  name={\ensuremath{\X_i}},
  description={\hfill\S\ref{sec: -limits},\ref{sec: +limits}},
  sort=xical
}

\newglossaryentry{Ycal}
{
  name={\ensuremath{\Y_i}},
  description={\hfill\S\ref{sec: ycal}},
  sort=yz
}

\newglossaryentry{At}
{
  name={\ensuremath{\mathbb{A}_{\tc}}},
  description={\hfill\S\ref{subsec: ddpa}},
  sort=at
}

\newglossaryentry{Atq}
{
  name={\ensuremath{\Atq}},
  description={\hfill\S\ref{subsec: ddpa}},
  sort=atq
}

\newglossaryentry{Q}
{
  name={\ensuremath{\dot{\mathbf{Q}}, \ddot{\mathbf{Q}}}},
  description={\hfill\S\ref{subsec: ddpa}},
  sort=q
}

\newglossaryentry{d}
{
  name={\ensuremath{d_+, d_+^*, d_-}},
  description={\hfill\S\ref{subsec: ddpa},\ref{sec: ddparep}},
  sort=d
}

\newglossaryentry{z}
{
  name={\ensuremath{z_i}},
  description={\hfill\S\ref{subsec: ddpa},\ref{sec: ddparep}},
  sort=z
}

\newglossaryentry{y}
{
  name={\ensuremath{y_i}},
  description={\hfill\S\ref{subsec: ddpa},}\ref{sec: ddparep},
  sort=y
}

\newglossaryentry{zeta}
{
  name={\ensuremath{\zeta_k}},
  description={\hfill\S\ref{sec: ddparep}},
  sort=zeta
}

\newglossaryentry{par}
{
  name={\ensuremath{\partial_k}},
  description={\hfill\S\ref{sec: ddpaP}},
  sort=del
}

\newglossaryentry{park}
{
  name={\ensuremath{\partial_k^* }},
  description={\hfill\S\ref{sec: ddpaP}},
  sort=del
}

\newglossaryentry{parmin}
{
  name={\ensuremath{\partial_k^- }},
  description={\hfill\S\ref{sec: d-}},
  sort=delmin
}

\newglossaryentry{B}
{
  name={\ensuremath{\Bo_n, \Bo_\infty}},
  description={\hfill\S\ref{sec: d-}},
  sort=bn
}

\newglossaryentry{HL}
{
  name={\ensuremath{P_\lambda[\Xb,\tc] }},
  description={\hfill\S\ref{sec: d-}},
  sort=pmu
}

\newglossaryentry{Phi}
{
  name={\ensuremath{\Phi_\bullet }},
  description={\hfill\S\ref{sec: isom}},
  sort=phi
}

\newglossaryentry{Phik}
{
  name={\ensuremath{\Phi_k}},
  description={\hfill\S\ref{sec: isom}},
  sort=phi
}

\title{The Stable Limit DAHA and the Double Dyck Path Algebra} 
\author{Bogdan Ion} 
\address{Department of Mathematics, University of Pittsburgh, Pittsburgh, PA 15260}
\email{bion@pitt.edu}
\author{Dongyu Wu}
\address{Department of Mathematics, University of Pittsburgh, Pittsburgh, PA 15260}
\address{Beijing Institute of Mathematical Sciences And Applications, Beijing 101408}
\email{dow16@pitt.edu, wudongyu@bimsa.cn}
\date{April 26, 2022}

\begin{document}

\begin{abstract}
We study the compatibility of the action of the DAHA of type GL with two inverse systems of polynomial rings obtained from the standard Laurent polynomial representations. In both cases, the crucial analysis is that of the compatibility of the action of the Cherednik operators. Each case leads to a representation of a limit structure (the +/- stable limit DAHA) on a space of almost symmetric polynomials in infinitely many variables (the standard representation). As an application, we show that the defining representation of the double Dyck path algebra arises from the standard representation of the  +stable limit DAHA.
\end{abstract}

\maketitle

\section{Introduction}

There are subtle phenomena in Macdonald theory (the positivity of the composition Kostka functions \cite{Kn}*{Conjecture 11.2}) that arise in the so-called stable limit, that is, they describe properties of objects that are associated to  
certain limits of polynomial functions when the number of variables approaches infinity. The object that provides the appropriate finite rank algebraic context in this case is the double affine Hecke algebra (DAHA) of type ${\rm GL}_k$. However, the only stable limit structure that was investigated more systematically is the stable limit of the spherical DAHA of type ${\rm GL}_k$, which is intimately related to several interconnected geometric contexts: the spherical subalgebra of the Hall algebra of the category of coherent sheaves on an elliptic curve \cites{BS, SV}, the convolution algebra in the equivariant K-theory of the Hilbert scheme of $\mathbb{A}^2$ \cite{SV2}, the  Feigin-Odesskii \cite{FO} shuffle algebra \cites{FT, SV2, Neg}.

In the case of the spherical DAHA, the stable limit structure arises as an inverse limit of the corresponding finite rank structure. The DAHA $\H_k$ of type ${\rm GL}_k$ do not form an inverse system and a conceivable residual limit structure  will not capture some relations that hold in particular finite ranks.
The basic idea towards understanding the residual limit structure is to undertake a study of the limiting behavior of individual elements (e.g. for a specific set of generators)  in an inverse system of representations and understand the algebraic structure defined by the limit operators. 

We pursue here this idea, as follows. Let $\P_k=\Rat(\tc,\qc)[x_1^{\pm 1},\dots,x_k^{\pm 1}]$ be the standard Laurent polynomial representation of $\H_k$; $\P_k$ is a faithful representation of $\H_k$. Let
$$\P_k^+=\Rat(\tc,\qc)[x_1,\dots,x_k]\quad 
\text{and}\quad \P_k^-=\Rat(\tc,\qc)[x_1^{-1},\dots,x_k^{-1}],$$
each being a faithful representation of a corresponding subalgebra of $\H_k$ denoted by $\H_k^+$ and, respectively, $\H_k^-$.  Both $(\P_k^+)_{k\geq 2}$ and $(\P_k^-)_{k\geq 2}$ form graded inverse systems, with structure maps that specialize the last variable to $0$. The corresponding graded inverse limits are denoted by $\P_\infty^{\pm }$; they are sometimes referred to as the rings of formal polynomials in the variables $x_i^\pm$, $i\geq 1$. For us, an important role is played by the graded subrings $\Pas^\pm\subset \P_\infty^\pm$ consisting of elements fixed by all simple transpositions of the variables, with the possible exception of finitely many. Following Knop \cite {Kn}, we refer to $\Pas^\pm$ as the \emph{almost symmetric modules} and to its elements as almost symmetric functions.

The action of the generators of $\H^\pm_k$ that act as Demazure-Lusztig operators or multiplications operators is immediately seen to be compatible with the inverse system;  each operator will consequently induce a limit operator acting on the corresponding graded inverse limit $\P_\infty^{\pm}$. Therefore, in both cases, the crucial analysis is that of the action of the Cherednik operators and of  their compatibility with the inverse system. 

The compatibility of the action of the Cherednik operators in $\H_k^-$ with the inverse system $\P_k^-$ was established in \cite{Kn}; the corresponding operators act on $\P_\infty^-$. Knop has also pointed out that  $\Pas^-\subset \P_\infty^-$ is stable under the action of the limit Cherednik operators and provides the natural context for the theory of nonsymmetric Macdonald functions (the stable limits of nonsymmetric Madonald polynomials). 

On the other hand, the analysis of the $\H_k^+$ stable limit action requires a more sophisticated approach. The action of the Cherednik operators in $\H_k^+$ is no longer compatible with the inverse system $\P_k^+$ and requires a careful investigation of their action in order to understand the obstructions to such a compatibility and what particular features might have a chance to manifest in the stable limit. We briefly discuss the main new developments here, postponing a more technical discussion for the later sections.

The failure of the compatibility of the action of the Cherednik operators  in $\H_k^+$ with the inverse system $\P_k^+$ can be witnessed at the level of their spectrum and allows for a precise identification of the obstruction. As it turns out, the Cherednik operator $Y_i$ is compatible with the inverse system $x_i\P_k^+$, inducing an action of the limit operator on $x_i\P_\infty^+$. While this might be satisfactory from the point of view of a single Cherednik operator, it does not lead to a non-zero common domain for all Cherednik operators $Y_i$, $i\geq 1$. It is therefore necessary to adjust the  action of $Y_i$ on a complement of $x_i\P_k^+$. We address this first difficulty uniformly, by introducing a new finite rank  algebraic structure $\Ht_k^+$, which we call \emph{the deformed DAHA}, and its standard representation on $\P_k^+$. The actions of the Cherednik operator $Y_i$  and the deformed Cherednik operator $\widetilde{Y}_i$  coincide on $x_i\P_k^+$. 

Nevertheless, $\Ht_k^+$ is a more complicated structure, and, most notably, the deformed Cherednik operators no longer commute. Their action is still not compatible with the inverse system, but they satisfy instead a more subtle compatibility. To express this compatibility we formulate a concept of limit that takes into consideration not only the inverse system, but also the $\tc$-adic topology. This allows us to define limit operators $\widetilde{Y}_i$ acting, not on $\P_\infty^+$, but only on $\Pas^+$, attesting again to the canonical nature of the almost symmetric module. On $x_i\Pas^+$ the action of the limit operator  $\widetilde{Y}_i$ coincides with the action of  the limit Cherednik operator $Y_i$.
This analysis occupies most of \S\ref{sec: +stablelimit}, culminating with the proof of Proposition \ref{prop: limit action}.

We relate the action of the limit operators on $\Pas^+$ to the following algebraic structure.  Let  $\H^+$ be  the $\mathbb{Q}(\tc,\qc)$-algebra    generated by the elements $T_i$,$X_i$, and $Y_i$, $i\geq 1$, satisfying  the following relations:
  \begin{subequations}%\label{sdaha}
        \begin{equation}%\label{T relations}
          \begin{gathered}
          T_{i}T_{j}=T_{j}T_{i}, \quad |i-j|>1,\\
          T_{i}T_{i+1}T_{i}=T_{i+1}T_{i}T_{i+1}, \quad i\geq 1,
          \end{gathered}
        \end{equation}
        \begin{equation}%\label{Quadratic}
                  (T_{i}-1)(T_{i}+\tc)=0, \quad i\geq 1,
        \end{equation}
        \begin{equation}%\label{X relations}
            \begin{gathered}
                \tc T_i^{-1} X_i T_i^{-1}=X_{i+1}, \quad  i\geq 1\\
                T_{i}X_{j}=X_{j}T_{i}, \quad  j\neq i,i+1,\\
                X_i X_j=X_j X_i,\quad i,j\geq 1,
            \end{gathered}
        \end{equation}
         \begin{equation}%\label{Y relations}
            \begin{gathered}
                \tc^{-1} T_i Y_i T_i=Y_{i+1}, \quad i\geq 1\\
                T_{i}Y_{j}=Y_{j}T_{i}, \quad  j\neq i,i+1,\\
                Y_i Y_j=Y_j Y_i, \quad i,j\geq 1,
            \end{gathered}
        \end{equation}
                \begin{equation}%\label{XY cross relations}
            Y_1 T_1 X_1=X_2 Y_1T_1.
        \end{equation}
    \end{subequations}
We call $\H^+$ the \sp limit DAHA. There is a corresponding \sm limit DAHA $\H^-$ and a canonical $\mathbb{Q}(\tc,\qc)$-algebra anti-isomorphism between $\H^+$ and $\H^-$. As a result, we use the terminology \emph{stable limit DAHA} to refer to either of them, depending on the context. Both algebras are closely related to the inductive limit of  the group algebras of the braid groups $\mathcal{B}_k$, of $k$ distinct points on the punctured torus. We refer to \S \ref{sec: braid} for the precise relationship.

The action of the limit operators on $\Pas^-$ can be easily seen to define a representation of $\H^-$. Our first main result is the following (Theorem \ref{thm: +standardrep}).
\begin{state} \label{state: A}
The action of the limit operators on $\Pas^+$ defines a representation of $\H^+$.
\end{state}
We call the representation of $\H^{\pm}$ on $\Pas^{\pm}$ the standard representation of the stable limit DAHA. The main difficulty in establishing this result is the proof of the commutativity of the limit $\widetilde{Y}_i$ operators. On this account, the commutativity of the Cherednik operators that was lost by deforming their action on $\P_k^+$ is restored in the limit. The action of the operators $T_i$, $X_i$, and $Y_i\in \H^+$ on $\Pas^+$ is denoted by $\T_i$, $\X_i$, and $\Y_i$, respectively.

One aspect that underscores the special nature of the positivity phenomenon in Macdonald theory is the fact that the work around the former Macdonald positivity conjecture (first proved by Haiman \cite{HaiHil} in the equivalent formulation of the $n!$ conjecture) and the closely related work on the ring of diagonal harmonics \cite{HaiVan}, relies on the geometry of Hilbert schemes of points in $\mathbb{A}^2$ and, so far, a workable explanation based on DAHA technology -- which was successful in settling all the other Macdonald conjectures -- is missing. Another apparent dissimilarity, is the fact that the representation theory relevant for the $n!$ conjecture is that of the symmetric group $S_n$, while all the other Macdonald conjectures are addressing phenomena rooted in the representation theory of reductive groups. 

Following the work in \cite{HaiVan}, a precise conjecture (the so-called Shuffle Conjecture) was formulated in \cite{HHLRU} (see also \cite{Hag}) for the monomial expansion of the bi-graded Frobenius characteristic of the ring of double harmonics. This conjecture, and a number of further refinements and generalizations were recently proved \cites{CM, Me}. In the process, Carlsson and Mellit revealed a new algebraic structure (the double Dyck path algebra $\Aqt$) together with a specific representation $V_\bullet$, which, aside from governing the combinatorics relevant for the  Shuffle Conjecture, exhibits striking structural similarities with the family of DAHA of type ${\rm GL}_k$, $k\geq 2$.  In  \cite{CGM} a connection was established between $\Aqt$ and the geometry of Hilbert schemes; it shows that the representation $V_\bullet$ is closely connected with a geometric action in the equivariant K-theory of certain smooth strata in the parabolic flag Hilbert schemes of points in $\mathbb{A}^2$. This action is akin to the action of stable limit spherical DAHA on the equivariant K-theory of the Hilbert schemes of points in $\mathbb{A}^2$  \cite{SV}. We refer to Definition \ref{def: ddpa}, Proposition \ref{prop: ddpa}, and Theorem \ref{thm: CMrep} for the definition of $\Atq$, $V_\bullet$, and other structural details.

The structural similarities between $\Aqt$ and the family of DAHA of type ${\rm GL}_k$ are therefore important for connecting the geometry of Hilbert schemes with the DAHA. The algebra $\Aqt$ is a quiver algebra and its representation $V_\bullet=(V_k)_{k\geq 0}$ is a quiver representation; the subalgebra of $\Aqt$ generated by certain elements that correspond to loops that act at node $k$ resembles quite closely the DAHA of type ${\rm GL}_k$. In fact, the generators in such a subalgebra satisfy all the expected DAHA relations except for one (see discussion  in \cite{CM}*{pg. 693--694}). However, it was expected that an explicit relationship with the DAHA exists. The main difficulty in establishing such a connection rests in explaining the relationship between the Cherednik operators and the operators $z_i$ in $\Aqt$.

We establish a direct connection between the standard representation of the \sp limit DAHA and the double Dyck path algebra. To facilitate the comparison and not deviate from the traditional notation for parameters in double affine Hecke algebras, we will swap the role of the parameters $\qc$ and $\tc$ in the original definition of the double Dyck path algebra. Therefore, the connection we establish is between the \sp limit DAHA and the algebra $\Atq$. Specifically, we show that the standard representation of $\H^+$ can be used to construct a quiver representation of $\Atq$. The standard representation $\Pas^+$ has a natural filtraton $\P(k)^+$, ${k\geq 0}$; the subspace $\P(k)^+$ consists of the elements of $\Pas^+$ that are symmetric in the variables $x_i$, $i>k$. We denote by $\P_\bullet$ the complex of vector spaces $(\P(k)^+)_{k\geq 0}$. The vector spaces $\P(k)^+$ form an direct system and $\Pas^+$ is their direct limit. Our second main result is the following (Theorem \ref{thm: quiverrep}).

\begin{state}\label{state: B}
There is a $\Atq$ action on $\P_\bullet$ such all generators of $\Atq$ corresponding to loops in Definition \ref{def: ddpa} (i.e. $T_i$, $y_i$, $z_i$) act as $\T_i$, $\X_i$, $\Y_i$.
\end{state}

In fact, this representation turns out to be precisely the defining representation of $\Atq$ (Theorem \ref{thm: isom})

\begin{state}\label{state: C}
The representations of $\Atq$ on $\P_\bullet$ and $V_\bullet$ are isomorphic.
\end{state}

Since the notation in \cite{CM} for the elements of the double Dyck path algebra that are relevant for its comparison with the action of $\H^+$ does not match the more traditional notation used in the DAHA literature, we include below,
for the reader's convenience,  the conversion table between notation used here  and the notation used in \cite{CM} and \cite{Me}:

\begin{table}[h!]
  \begin{center}
    %\caption{Your first table.}
    \label{tab:table1}
    \begin{tabular}{c|c|c|c|c|c|c|c|c|c|c} % <-- Alignments: 1st column left, 2nd middle and 3rd right, with vertical lines in between
      & $\tc$ & $\qc$ & $\P(k)^+$ & $\P_\bullet$ & $\X_i$ & $\Y_i$ & $\T_i$  & $\partial$ & $\partial^*$ & $\partial^-$\\
      \hline
      \cite{CM} & q & t  & $V_k$ & $V_*$ & $y_i$ & $z_i$ & $T_i$ & $d_+$ & $d_+^*$ & $d_-$\\
    \end{tabular}
  \end{center}
\end{table}

It is important to note that all the generators of $\Atq$ corresponding to loops act \emph{globally} in $\P_\bullet$, by which we mean that  their action on each $\P(k)^+$ is the restriction of their action on $\Pas^+$. The vector spaces $V_k$, $k\geq 0$, also form a direct system but the action of the elements $z_i\in \Atq$ act only \emph{locally}, by which we mean that the action is \emph{not compatible} with the structure maps of the direct system. This obscures the fact that the representation $V_\bullet$ arises from a more fundamental representation on the direct limit of $V_k$, $k\geq 0$. In some sense, the isomorphism between $\P_\bullet$ and $V_\bullet$ straightens this out this apparent incompatibility and explains how the limit Cherednik operators lead to the operators $z_i$ in $\Atq$.

In \cite{CM}*{pg. 694} (see also \cite{CGM}*{Remark 2.6}) there is one phrase commenting on a possible algebra isomorphism between $e_k\Atq e_k$ (the subalgebra of $\Atq$ generated by loops based at node $k$) and a \qq{partially symmetrized}  version of the stable limit \emph{spherical} DAHA, still to be defined for $k>0$. Our results, Theorem \ref{state: B} and Theorem \ref{state: C}, establish a different relationship, between different objects, namely between the standard representations of $\H^+$ and $\Atq$.

Our construction of the standard representations of the stable limit DAHA opens a number of immediate questions on the spectral theory of the limit Cherednik operators as well as on the structure of these representations. A second set of questions is related the possible applications of the stable limit DAHA -- through the path opened by the work of Carlsson and Mellit -- to the combinatorics and geometry of Hilbert schemes. We hope to pursue these questions in  future publications.

\begin{ack}
The work of BI was partially supported by the Simons Foundation grant 420882. 
\end{ack}

\section{Notation}

\subsection{}\label{sec: not1}
 We denote by $\Xb$ an infinite alphabet $x_1,x_2,\dots$ and by $\Sym[\Xb]$ the ring of symmetric functions in $\Xb$. The field or ring of coefficients $\mathscr{K}\supseteq \Rat$ will depend on the context. For any $k\geq 1$, 
we denote by $\overline{\mathbf{X}}_k$ the finite alphabet $x_1,x_2,\dots, x_k$ and by  ${\mathbf{X}}_k$ the infinite alphabet $x_{k+1},x_{k+2},\dots$. $\Sym[{\mathbf{X}}_k]$ will denote the ring of symmetric functions in 
${\mathbf{X}}_k$. 
Furthermore, for any $1\leq k\leq m$, we denote by $\overline{\mathbf{X}}_{[k,m]}$ the finite alphabet $x_k,\dots, x_m$. As usual, we denote by \gls{hn} (or $h_n[\overline{\mathbf{X}}_k]$, or 
$h_n[\mathbf{X}_k]$, 
or $h_n[\overline{\mathbf{X}}_{[k,m]}]$) the $n$-{th} complete symmetric functions (or polynomials) in the indicated alphabet (\gls{Xb} \gls{Xbar}), by \gls{pn} (or $p_n[\overline{\mathbf{X}}_k]$, or $p_n[\mathbf{X}_k]$, or 
$p_n[\overline{\mathbf{X}}_{[k,m]}]$) 
the $n$-th power sum symmetric functions (or polynomials), and  by \gls{en} (or $e_n[\overline{\mathbf{X}}_k]$, or $e_n[\mathbf{X}_k]$, or 
$e_n[\overline{\mathbf{X}}_{[k,m]}]$) 
the $n$-th elementary symmetric functions (or polynomials). The symmetric function $p_1[\Xb]=h_1[\Xb]=e_1[\Xb]$ is also denoted by $\Xb=x_1+x_2+\cdots$. For a partition $\lambda$, \gls{m}  (or $m_\lambda[\overline{\mathbf{X}}_k]$, or 
$m_\lambda[\mathbf{X}_k]$, denotes the monomial symmetric function (or polynomial) in the indicated alphabet.

\subsection{}\label{sec: not2}
Any action of the  monoid $(\Z_{>0},\cdot)$ on the ring $\mathscr{K}$ extends to a canonical action by $\Rat$-algebra morphisms on $\Sym[\Xb]$. The morphism corresponding to the action of $n\in \Rat_{>0}$ is denoted by 
$\mathfrak{p}_n$ 
and is defined by
$$
\mathfrak{p}_n\cdot p_k[\Xb]=p_{nk}[\Xb], \quad k\geq 1.
$$
In our context $\mathscr{K}$ will be a ring of polynomials or a field of fractions generated by some finite set of parameters, the action of $(\Z_{>0},\cdot)$ on $\mathscr{K}$ is $\Rat$-linear, and $\mathfrak{p}_n$ acts on 
parameters by raising them to the $n$-th power. For example, if $\mathscr{K}=\Rat(\tc,\qc)$ then, $\mathfrak{p}_n\cdot \tc=\tc^n, ~\mathfrak{p}_n\cdot \qc=\qc^n$.

Let $R$ be a ring with an action of $(\Z_{>0},\cdot)$ by ring morphisms. Any ring morphism $\varphi: \Sym[\Xb]\to R$ that is compatible with the action of $(\Z_{>0},\cdot)$ is uniquely determined by the image of $p_1[\Xb]=\Xb$. 
The image of $F[\Xb]\in \Sym[\Xb]$ through $\varphi$ is usually denoted by  $F[\varphi(\Xb)]$ and called the plethystic evaluation (or substitution) of $F$ at $\varphi(\Xb)$.

The plethystic exponential \gls{Exp} is defined as
$$\textrm{Exp}[\Xb]=\sum_{n=0}^{\infty}h_n[\Xb]=\exp\left(\sum_{n=1}^\infty \frac{p_n[\Xb]}{n}\right).$$

\subsection{} \label{sec: esf}
We will use some symmetric polynomials that are related to the complete homogeneous symmetric functions via plethystic substitution. More precisely, let   \gls{hnf}  be the symmetric polynomial obtained from  the symmetric function $h_n[(1-\tc)\Xb]$ by specializing to $0$ the elements of the alphabet ${\mathbf{X}}_k$. The corresponding notation applies to $h_n[(\tc-1)\Xb]$ and other plethystic substitutions.

For example, from the Newton identities for the complete homogeneous symmetric functions, we can easily see that (for $x=x_1$)
\begin{equation}\label{eq2}
h_n[(1-\tc)x]=(1-\tc)x^n, \quad h_n[(\tc-1)x]=\tc^{n-1}(\tc-1)x^n, n\geq 1,
\end{equation}
and  $h_0[(1-\tc)x]=h_0[(\tc-1)x]=1$. The classical convolution formula $\displaystyle h_n[\Xb]=\sum_{i+j=n} h_i[\Xb] h_j[\Xb]$ has the following counterparts 

\begin{equation}\label{eq3}
h_n[(1-\tc)\overline{\mathbf{X}}_k]=\sum_{i+j=n} h_i[(1-\tc)\overline{\mathbf{X}}_{k-1}]h_j[(1-\tc)x_k]
\end{equation}
and
\begin{equation}\label{eq4}
h_n[(1-\tc)\overline{\mathbf{X}}_{k-1}]=\sum_{i+j=n} h_i[(1-\tc)\overline{\mathbf{X}}_{k}]h_j[(\tc-1)x_k].
\end{equation}
We will make use of such formulas applied to similar situations -- for example, with $\overline{\mathbf{X}}_{[k+1,m]}$, $m>k$, replacing $\overline{\mathbf{X}}_{k}$ in the formulas above. 

\subsection{} \label{sec: not4}

For any $k\geq 1$,  let $ \gls{Pn}=\Rat(\tc,\qc)[x_1^{\pm 1},\dots,x_k^{\pm 1}]$ be the ring of Laurent polynomials in the variables $x_1,\dots, x_k$.  The symmetric group $S_k$ acts on $\P_k$ by permuting the variables. We denote by \gls{sn} the simple transposition that interchanges $x_i$ and $x_{i+1}$ and is fixing all the other variables. The polynomial subrings 
$$\P_k^+=\Rat(\tc,\qc)[x_1,\dots,x_k]\quad 
\text{and}\quad \P_k^-=\Rat(\tc,\qc)[x_1^{-1},\dots,x_k^{-1}]$$
are stable under the action of $S_k$.

\subsection{}\label{sec: maps}

Let $$\gls{pik}:\P^+_k\rightarrow \P^+_{k-1},\quad \text{and } \quad \gls{ik}: \P_{k-1}^{+}\rightarrow \P_{k}^{+} $$
be the evaluation morphism that maps $x_{k}$ to $0$ and, respectively, the canonical inclusion.  The definitions and facts that we discuss below have  straightforward analogues associated to the sequence of rings $\P_k^-$, $k\geq 1$.  To avoid introducing additional notation, we will use $\pi_k$ and \gls{ik} for the corresponding morphisms  between the rings $\P_k^-$ and $\P_{k-1}^-$ (to be clear, in this case, the evaluation morphism $\pi_k$ maps $x_k^{-1}$ to $0$), with the hope that the necessary distinction will be clear from the context (in fact, the objects associated with $\P_k^-$ will appear only in \S\ref{sec: -stable}). We adopt the same convention for all the other morphisms considered in this sub-section: $\iota_{n,k}$, \gls{Pik},  \gls{Ik},  \gls{Jk},   \gls{J}.

The rings $\P_k^\pm$, $k\geq 1$ form an graded inverse system. We will use the notation  \gls{Pinf} for the graded inverse limit ring $\displaystyle 
\lim_{\longleftarrow}\P_k^{\pm}$. The graded inverse limit rings are sometimes referred to in the literature as the rings of formal polynomials in the variables $x_i^\pm$, $i\geq 1$. We denote by $\Pi_k: \displaystyle \lim_{\longleftarrow}\P_k^{\pm}\to \P_k^\pm$ the canonical morphism.

The inductive limit $\displaystyle \lim_{{\longrightarrow}} \P_k^+$ is canonically isomorphic with the ring $\mathbb{Q}(\tc,\qc)[x_1,x_2,\dots ]$ of polynomials in infinitely many variables (and similarly for $\displaystyle \lim_{{\longrightarrow}} \P_k^-$). As 
with the projective limit,  the rings $\displaystyle \lim_{\substack{\longrightarrow \\ k\geq n}} \P_k^\pm$ and  $\displaystyle \lim_{{\longrightarrow}} \P_k^\pm$ are canonically isomorphic. We denote by $I_k: \displaystyle 
\P_k^\pm\to\lim_{{\longrightarrow}} \P_k^\pm$ the canonical morphism.

 The following diagram is commutative:
\[
\begin{tikzcd}
  \P_{k-1}^{\pm} \arrow[r] \arrow[d, "\iota_k"'] & \P_{k-1}^{\pm} \\
  \P_{k}^{\pm} \arrow[r]                 & \P_{k}^{\pm}               \arrow[u, "\pi_k"]   
\end{tikzcd}
\]
where each horizontal map represents the identity map. For a fixed $n\geq 1$,  denote by $\iota_{n,k}: \P_n^\pm\to \P_k^\pm$, $n\leq k$, the canonical inclusion. The sequence of maps $\iota_{n,k}: \P_n^\pm\to \P_k^\pm$, $k\geq n$ is compatible 
with the structure maps $\pi_k$. Therefore, they induce a morphism
$$
\P_n^\pm\to  \lim_{\substack{\longrightarrow \\ k\geq n}} \P_k^\pm\cong \displaystyle \lim_{{\longrightarrow }} \P_k^\pm.
$$
Furthermore, these maps are compatible with the structure maps $\iota_n$ and therefore induce a morphism
$$
J: \lim_{{\longrightarrow}} \P_k^\pm \to \lim_{\longleftarrow}\P_k^{\pm}.
$$
By construction, $\Pi_kJI_k: \P_k^\pm\to \P_k^\pm$ is the identity function. We denote $J_k=JI_k: \P_k^\pm\to \P_\infty^\pm$.

\subsection{}\label{sec: p(k)}

For any $n\geq 1$, a sequence of operators $A_k:\P_k^\pm\to \P_k^\pm$, $k\geq n$, compatible with the inverse system induces a (limit) operator $A: \P_\infty^\pm\to\P_\infty^\pm$.  We have $A_k=\Pi_k A J_k$. For example, the sequence $A_k=s_n$, $k\geq n$, given by the action of the simple transposition $s_n$, induces a limit operator $s_n$ acting on $\P_\infty^\pm$. In turn, this leads to an action of the infinite symmetric group $S(\infty)$ (the inductive limit of $S_k$, $k\geq 1$) on $\P_\infty^\pm$.

For any $k\geq 0$, denote,
$$\gls{P(n)}=\{F\in \P_{\infty}^{\pm}\ |\ s_i F=F,~ \text{for all } i>k \}.$$
From the definition it is clear that ${\P}(k)^{\pm}\subset{\P}(k+1)^{\pm}$. Also, ${\P}(0)^{+}$ is the ring of symmetric functions  $\Sym[\Xb]$, and, more generally, for any $k\leq 1$,
 the multiplication map  $$ \P_k^{+}\otimes \Sym[\mathbf{X}_k]  \cong\P(k)^+$$ is an algebra isomorphism.

\subsection{}\label{sec: pas}
An important role will be played by the graded subrings $\Pas^\pm\subset \P_\infty^\pm$,  defined as the inductive limit of the spaces $\P(k)^\pm$:
$${\P}_{\textnormal{as}}^{\pm}=\bigcup_{k\geq 0}{\P}(k)^{\pm}.$$
More concretely, an element of $\Pas^\pm$ must be fixed by all simple transpositions with the possible exception of finitely many. Following Knop \cite {Kn}, we refer to \gls{Pas} as the \emph{almost symmetric modules}.

%----------------------------------------------------------------------------------------

\section{Preliminaries}

\subsection{Affine Hecke Algebras}\label{sec: AHA}
\begin{dfn}\label{def: AHA}
  The affine Hecke algebra \gls{Ak} of type $GL_k$ is the $\mathbb{Q}(\tc)$-algebra generated by 
$$T_1,\dots,T_{k-1},X_1^{\pm1},\dots,X_k^{\pm1}$$
satisfying the following relations:
    \begin{subequations}\label{AHA}
        \begin{equation}\label{T relation ii}
          \begin{gathered}
          T_{i}T_{j}=T_{j}T_{i}, \quad |i-j|>1,\\
          T_{i}T_{i+1}T_{i}=T_{i+1}T_{i}T_{i+1}, \quad 1\leq i\leq k-2,
          \end{gathered}%\tag{3.1.a}
          \end{equation}
          \begin{equation}\label{quadratic}
                    (T_{i}-1)(T_{i}+\tc)=0, \quad 1\leq i\leq k-1,
          \end{equation}
        \begin{equation}\label{X relation ii}
            \begin{gathered}
                \tc T_i^{-1} X_i T_i^{-1}=X_{i+1}, \quad 1\leq i\leq k-1\\
                T_{i}X_{j}=X_{j}T_{i},\quad  j\neq i,i+1,\\
                X_i X_j=X_j X_i \quad 1\leq i,j\leq k.
            \end{gathered}%\tag{3.1.b}
        \end{equation}
    \end{subequations}
\end{dfn}

The definition above is based on the Bernstein presentation of the affine Hecke algebra $\AHA_k$. The following presentation of $\AHA_k$ will also be used.

\begin{prop}
  The affine Hecke algebra $\AHA_k$ is generated by 
$$\widetilde{T}_0,T_1,\dots,T_{k-1},\widetilde{\omega}_k^{\pm 1}$$
satisfying the relations:
    \begin{subequations}\label{AHA1}
        \begin{equation}\label{T relation iii}
          \begin{gathered}
          T_{i}T_{j}=T_{j}T_{i}, \quad |i-j|>1,\\
          T_{i}T_{i+1}T_{i}=T_{i+1}T_{i}T_{i+1}, \quad 1\leq i\leq k-2,\\
          (T_{i}-1)(T_{i}+\tc)=0, \quad 1\leq i\leq k-1,
            \end{gathered}%\tag{3.2.a}
        \end{equation}
        \begin{equation}
        \begin{gathered}
          T_{i} \widetilde{T}_0=\widetilde{T}_0 T_{i}, \quad 2\leq i\leq k-1,\\
          \widetilde{T}_0 T_{1}\widetilde{T}_0=T_{1}\widetilde{T}_0 T_{1}, \\
          (\widetilde{T}_0-1)(\widetilde{T}_0+\tc)=0, 
          \end{gathered}%\tag{3.2.a}
        \end{equation}
        \begin{equation}\label{omega relation}
            \widetilde{\omega}_k T_i \widetilde{\omega}_k^{-1}=T_{i-1} \textrm{ for }1\leq i\leq k-1,\quad \widetilde{\omega}_k \widetilde{T}_0 \widetilde{\omega}_k^{-1}=T_{k-1}.%\tag{3.2.b}
        \end{equation}
    \end{subequations}
\end{prop}

\begin{rmk}\label{rem: otilderels}
$\widetilde{T}_0$, \gls{tomegak},  and $X_1,\dots,X_k$ are related by
$$\widetilde{\omega}_k=\tc^{1-k} 
T_{k-1}\dots T_{1}X_{1}^{-1}=
%\tc^{k-1}
X_k^{-1}T_{k-1}^{-1}\dots T_{1}^{-1},$$
$$\widetilde{T}_0=\widetilde{\omega}_k^{-1}T_{k-1}\widetilde{\omega}_k=\widetilde{\omega}_k T_1\widetilde{\omega}_k^{-1} =\tc^{k-1}X_1X_k^{-1}T_1^{-1}\dots T_{k-1}^{-1}\dots T_1^{-1}.$$
\end{rmk}

\begin{note}
We denote by $\AHA_k^+$ the subalgebra of $\AHA_k$ generated by $T_i$, $i\leq k-1$, and $X_i$, $i\leq k$, or equivalently, by $T_i$, $i\leq k-1$, and $\widetilde{\omega}_k^{-1}$. Similarly, we denote by $\AHA_k^-$ the subalgebra of 
$\AHA_k$ 
generated by $T_i$, $i\leq k-1$, and $X^{-1}_i$, $i\leq k$, , or equivalently, by $T_i$, $i\leq k-1$, and $\widetilde{\omega}_k$.
\end{note}

%----------------------------------------------------------------------------------------

\subsection{Double Affine Hecke Algebras}\label{sec: DAHA}

\begin{dfn}\label{def: DAHA}
The double affine Hecke algebra (DAHA)  \gls{Hk}, $k\geq 2$, of type $GL_{k}$ is the $\mathbb{Q}(\tc,\qc)$-algebra generated by $T_0,T_1,\dots,T_{k-1}$, $X_1^{\pm1},\dots,X_k^{\pm1}$, and $\omega_k^{\pm 1}$ satisfying \eqref{T 
relation ii}, \eqref{quadratic}, \eqref{X relation ii} and the following relations:
    \begin{subequations}\label{DAHA2}
        \begin{equation}\label{T relation iv}
            \begin{gathered}
	          T_{i} T_0=T_0 T_{i}, \quad 2\leq i\leq k-1,\\
	          T_0 T_{1}T_0=T_{1}T_0 T_{1},
            \end{gathered}
        \end{equation}
        \begin{equation}
        	          (T_0-1)(T_0+\tc)=0, 
        \end{equation}
        \begin{equation}\label{omega2 relation}
	\begin{gathered}
            \omega_k T_i \omega_k^{-1}=T_{i-1}, \quad 2\leq i\leq k-1,\\
	 \omega_k T_1 \omega_k^{-1}=T_0, \quad \omega_k T_0 \omega_k^{-1}=T_{k-1},
	\end{gathered}
        \end{equation}
        \begin{equation}\label{X-omega2 cross relation}
          \begin{gathered}
	 \omega_k X_{i+1}\omega_k^{-1}=X_i,  \quad 1\leq i\leq k-1,\quad \omega_k X_1 \omega_k^{-1}=\qc^{-1}X_k.
          \end{gathered}
        \end{equation}
    \end{subequations}
\end{dfn}

The double affine Hecke algebras were first introduced,  in greater generality, by Cherednik \cite{Ch}. Under one of the possible sets of conventions, the definition of the DAHA of type $GL_k$ can be found, for example, in 
\cite{Ch-book}*{\S3.7}. 
The presentation in Definition \ref{def: DAHA} is consistent with the conventions in \cite{IS}. We will also make use of the following equivalent presentation of $\H_k$.

\begin{prop}\label{DAHA}
The algebra $\H_k$, $k\geq 2$, is the $\mathbb{Q}(\tc,\qc)$-algebra generated by the elements $T_1,\dots,T_{k-1}$, $X_1^{\pm1},\dots,X_k^{\pm1}$, $Y_1^{\pm1},\dots,Y_k^{\pm1}$ satisfying \eqref{T relation ii}, \eqref{quadratic}, 
\eqref{X relation ii} and the following relations:
    \begin{subequations}\label{DAHA1}
        \begin{equation}\label{Y relation}
            \begin{gathered}
                \tc^{-1} T_i Y_i T_i=Y_{i+1}, \quad 1\leq i\leq k-1\\
                T_{i}Y_{j}=Y_{j}T_{i}, \quad  j\neq i,i+1,\\
                Y_i Y_j=Y_j Y_i \quad 1\leq i,j\leq k,
            \end{gathered}
        \end{equation}
        \begin{equation}\label{XY cross relation}
            Y_1 T_1 X_1=X_2 Y_1T_1,
        \end{equation}
        \begin{equation}\label{det relation}
        Y_1 X_1\dots X_k= \qc X_1\dots X_k Y_1.
        \end{equation}
    \end{subequations}
\end{prop}

\begin{rmk}
$Y_1,\dots,Y_k$ and \gls{omegak} are related by
$$Y_i = \tc^{k+1-i}T_{i-1} \dots T_{1}\omega_k^{-1} T_{k-1}^{-1}\dots T_{i}^{-1}.$$
\end{rmk}

The presentation in Proposition \ref{DAHA} can be derived from Definition \ref{def: DAHA}. See also \cite{SV}*{\S2.1}; the generators $X_i$, \gls{Y}, in \cite{SV} correspond to $X_i^{-1}, Y_i^{-1}$ in our notation, and the relation 
\cite{SV}*{(2.7)} 
is equivalent to \eqref{XY cross relation} modulo the relations \eqref{X relation ii} and \eqref{Y relation}.

\begin{note}
We denote by $\H_k^+$ the subalgebra of $\H_k$ generated by \gls{T}, $i\leq k-1$, and  \gls{X}, $Y_i$, $i\leq k$. Similarly, we denote by $\H_k^-$ the subalgebra of $\H_k$ generated by $T_i$,$i\leq k-1$, and  $X^{-1}_i$, $Y^{-1}_i$, 
$i\leq k$.
\end{note}

%%%%%%%%%%%%%%%%%%%%%%%%%%%%%%%%%%%%%%%%%%%%%%%%%%%%%%%%%
\subsection{} \label{sec: standardrep}
The representation below is called the standard representation of $\H_k$ \cite{Ch}*{Theorem 2.3}

\begin{prop}\label{laurent rep}
  The following formulas define a faithful representation of $\H_k$ on $\P_k$:
  \begin{align}\label{DAHA representation}
      \begin{split}
	    T_i f(x_1,\dots,x_k) &= s_i f(x_1,\dots,x_k) +(1-\tc)x_i\frac{1-s_i}{x_i-x_{i+1}}f(x_1,\dots,x_k), \\
	    \widetilde{\omega}_k f(x_1,\dots,x_k) &=  %\tc^{\frac{1-k}{2}} 
	    \tc^{1-k}T_{k-1}\dots T_{1}x_1^{-1} f(x_1,\dots,x_k),\\    
	    \omega_k f(x_1,\dots,x_k) &= f(\qc^{-1} x_k,x_1,\dots,x_{k-1}).
%    Y_i &\mapsto& \tc^{1-i}T_{i-1} \dotsT_{1}\omega^{-1} T_{k-1}^{-1}\dotsT_{i}^{-1},
      \end{split}%\tag{3.3.b}
  \end{align}
Furthermore, for any $1\leq i\leq k$, the elements $X_i$ act as left multiplication by $x_i$.
\end{prop}

By restriction, we obtain the corresponding standard representations of $\H^+_k$ and $\H^-_k$.
\begin{cor} 
The subspace  \gls{Pnpm} is stable under the action of  \gls{Hk}. The corresponding representation of $\H_k^\pm$ on $\P_k^\pm$ is faithful.
\end{cor}
All the standard representations are induced representations from a one-dimensional representation of the corresponding affine Hecke subalgebra.

\begin{conv}\label{convention}
There will be elements denoted by the same symbol that belong to several (often infinitely many) algebras. The  notation does not keep track of this information if it is implicit  from the context. When necessary, we will add  the 
superscript $(k)$ (e.g. \gls{Y(k)} $\in \H_k$) to make such information explicit.
\end{conv}

%%%%%%%%%%%%%%%%%%%%%%%%%%%%%%%%%%%%%%%%%%%%%%%%%%%%%%%%%
\section{The stable limit DAHA}

%%%%%%%%%%%%%%%%%%%%%%%%%%%%%%%%%%%%%%%%%%%%%%%%%%%%%%%%%

\subsection{} \label{sec: sDAHA}
We introduce a pair of closely related algebras with infinitely many generators, which we call stable limit DAHAs.

\begin{dfn}\label{def: sDAHA+}
Let  $\H^+$ be  the $\mathbb{Q}(\tc,\qc)$-algebra    generated by the elements $T_i$,$X_i$, and $Y_i$, $i\geq 1$, satisfying  the following relations:
  \begin{subequations}\label{sdaha}
        \begin{equation}\label{T relations}
          \begin{gathered}
          T_{i}T_{j}=T_{j}T_{i}, \quad |i-j|>1,\\
          T_{i}T_{i+1}T_{i}=T_{i+1}T_{i}T_{i+1}, \quad i\geq 1,
          \end{gathered}
        \end{equation}
        \begin{equation}\label{Quadratic}
                  (T_{i}-1)(T_{i}+\tc)=0, \quad i\geq 1,
        \end{equation}
        \begin{equation}\label{X relations}
            \begin{gathered}
                \tc T_i^{-1} X_i T_i^{-1}=X_{i+1}, \quad  i\geq 1\\
                T_{i}X_{j}=X_{j}T_{i}, \quad  j\neq i,i+1,\\
                X_i X_j=X_j X_i,\quad i,j\geq 1,
            \end{gathered}
        \end{equation}
         \begin{equation}\label{Y relations}
            \begin{gathered}
                \tc^{-1} T_i Y_i T_i=Y_{i+1}, \quad i\geq 1\\
                T_{i}Y_{j}=Y_{j}T_{i}, \quad  j\neq i,i+1,\\
                Y_i Y_j=Y_j Y_i, \quad i,j\geq 1,
            \end{gathered}
        \end{equation}
                \begin{equation}\label{XY cross relations}
            Y_1 T_1 X_1=X_2 Y_1T_1.
        \end{equation}
    \end{subequations}
We will call $\H^+$ the \sp limit DAHA.
\end{dfn}
Note that, as opposed to the corresponding elements of $\H_k$, the elements $X_i$, $Y_i$ are not  invertible in $\H^+$. Similarly, we define the \sm limit DAHA.

\begin{dfn}\label{def: sDAHA-}
Let  $\H^-$ be  the $\mathbb{Q}(\tc,\qc)$-algebra    generated by the elements $T_i$,$X^{-1}_i$, and $Y^{-1}_i$, $i\geq 1$, satisfying  the following relations:
  \begin{subequations}%\label{sdaha}
        \begin{equation}%\label{T relations}
          \begin{gathered}
          T_{i}T_{j}=T_{j}T_{i}, \quad |i-j|>1,\\
          T_{i}T_{i+1}T_{i}=T_{i+1}T_{i}T_{i+1}, \quad i\geq 1,
           \end{gathered}
        \end{equation}
         \begin{equation}
          (T_{i}-1)(T_{i}+\tc)=0, \quad i\geq 1,
        \end{equation}
        \begin{equation}%\label{X relations}
            \begin{gathered}
                \tc^{-1} T_i X_i^{-1} T_i=X_{i+1}^{-1}, \quad  i\geq 1\\
                T_{i}X^{-1}_{j}=X^{-1}_{j}T_{i}, \quad j\neq i,i+1,,\\
                X^{-1}_i X^{-1}_j=X^{-1}_j X^{-1}_i,\quad i,j\geq 1,
            \end{gathered}
        \end{equation}
         \begin{equation}%\label{Y relations}
            \begin{gathered}
                \tc T_i^{-1} Y^{-1}_i T^{-1}_i=Y^{-1}_{i+1}, \quad i\geq 1\\
                T_{i}Y^{-1}_{j}=Y^{-1}_{j}T_{i}, \quad j\neq i,i+1,\\
                Y^{-1}_i Y^{-1}_j=Y^{-1}_j Y^{-1}_i, \quad i,j\geq 1,
            \end{gathered}
        \end{equation}
                \begin{equation}%\label{XY cross relations}
            X^{-1}_1 T^{-1}_1 Y^{-1}_1=T^{-1}_1 Y^{-1}_1 X^{-1}_2.
        \end{equation}
    \end{subequations}
We will call $\H^-$ the \sm limit DAHA.
\end{dfn}

The map $$\gls{e}: \H^+\to \H^-$$ that sends $T_i$, $X_i$, $Y_i$ to $T_i$, $Y^{-1}_i$, $X^{-1}_i$, respectively,  extends to an anti-isomorphism of  $\mathbb{Q}(\tc,\qc)$-algebras. For this reason, we regard $\H^+$ and $\H^-$ as capturing 
the same structure and we will use the terminology \emph{stable limit DAHA} to refer to either of them, depending on the context. 

The terminology is justified by the fact that these structures arise from analyzing the stabilization phenomena for the standard representations of the algebras $\H^\pm_k$ as $k$ approaches infinity. We will construct natural representations of both 
\gls{H}  in this fashion. To point more directly  to his connection consider the following.

\begin{dfn}
For any $k\geq 2$, denote by $\H(k)^+$ the subalgebra of $\H^+$ generated by $T_i$, $i\leq k-1$, and $X_i$, $Y_i$, $1\leq i\leq k$; $\H(k)^-$ denotes the  subalgebra of $\H^-$ generated by $T_i$, $i\leq k-1$, and $X^{-1}_i$, $Y^{-1}_i$, $1\leq i\leq k$.
\end{dfn}

\begin{rmk}\label{rem: projection}
We have a canonical surjective  $\mathbb{Q}(\tc,\qc)$-algebra morphisms $ \gls{H(k)}\to \H^\pm_k$.
\end{rmk}

%%%%%%%%%%%%%%%%%%%%%%%%%%%%%%%%%%%%%%%%%%%%%%%%%%%%%%%%%
\subsection{}\label{sec: braid} The stable limit DAHA is also closely related to the inductive limit of the braid groups $\mathcal{B}_k$, of $k$ distinct points on the punctured torus. Indeed, we have the following presentation of $\mathcal{B}_k$ 
\cite{Bel}*{Theorem 
1.1} (see also \cite{Me}*{Theorem 5.1}).

\begin{thm}\label{braid} For $k\geq 2$, the group $\mathcal{B}_k$ is generated by the elements $\sigma_i$, $1\leq i\leq k-1$, and $X_i$, $Y_i$, $1\leq i\leq k$ satisfying the following relations:
  \begin{subequations}
        \begin{equation}
          \begin{gathered}
          \sigma_{i}\sigma_{j}=\sigma_{j}\sigma_{i}, \quad |i-j|>1,\\
          \sigma_{i}\sigma_{i+1}\sigma_{i}=\sigma_{i+1}\sigma_{i}\sigma_{i+1}, \quad1\leq  i\leq k-2,
          \end{gathered}
        \end{equation}
        \begin{equation}
            \begin{gathered}
                \sigma_i^{-1} X_i \sigma_i^{-1}=X_{i+1}, \quad  1\leq i\leq k-1\\
                \sigma_{i}X_{j}=X_{j}\sigma_{i}, \quad j\neq i,i+1,\\
                X_i X_j=X_j X_i,\quad 1\leq i,j\leq k,
            \end{gathered}
        \end{equation}
         \begin{equation}
            \begin{gathered}
                \sigma_i Y_i \sigma_i=Y_{i+1}, \quad1\leq  i\leq k-1\\
                \sigma_{i}Y_{j}=Y_{j}\sigma_{i}, \quad j\neq i,i+1,\\
                Y_i Y_j=Y_j Y_i, \quad 1\leq i,j\leq k,
            \end{gathered}
        \end{equation}
                \begin{equation}
            Y_1 \sigma_1 X_1=X_2 Y_1\sigma_1.
        \end{equation}
    \end{subequations}
\end{thm}

The notation here and the one in \cite{Bel} are related as follows: $X_1=b_1$, $Y_1=a_1^{-1}$. 

Denote by $\mathcal{B}_k^+$ the sub-monoid of $\mathcal{B}_k$ generated by $\sigma_i,\sigma_i^{-1}$, $1\leq i\leq k$, and $X_i$, $Y_i$, $1\leq i\leq k$. The canonical inclusion maps between the monoids $\mathcal{B}^+_k$ 
consitute a direct system. The direct limit monoid is denoted by $\mathcal{B}^+_\infty$. A comparison between the relations in Theorem \ref{braid} and those in Definition \ref{def: sDAHA+} (with $\sigma_i$ mapping to 
$\tc^{-\frac{1}{2}}T_i$, 
$i\geq 1$) shows that $\H^+$ is the quotient of the monoid ring of $\mathcal{B}^+_\infty$ by the ideal generated by the quadratic relations (\ref{Quadratic}). A similar relationship connects $\H^-$ and the monoid 
$\mathcal{B}^-_\infty$, 
the inductive limit of the sub-monoids $\mathcal{B}_k^-$ of $\mathcal{B}_k$ generated by $\sigma_i,\sigma_i^{-1}$, $1\leq i\leq k$, and $X_i^{-1}$, $Y_i^{-1}$, $1\leq i\leq k$.

%%%%%%%%%%%%%%%%%%%%%%%%%%%%%%%%%%%%%%%%%%%%%%%%%%%%%%%%%
\subsection{}\label{sec: sl2}
%For any non-zero $a,b\in \mathbb{Q}(\tc,\qc)$ we can define the $\mathbb{Q}(\tc,\qc)$-algebra automorphism
%$$\sigma_{a,b}: \H^+\to \H^+$$
%that acts as scales all $X_i$ by $a$, scales all $Y_i$ by $b$ and acts as identity on all the other generators.

Let  $$\gls{iota}: \H^+\to \H^+$$ 
be the $\mathbb{Q}$-algebra automorphism that sends $T_i$ to $T_i^{-1}$, swaps $X_i$ and $Y_i$, and inverts the parameters $\tc$ and $\qc$. It is easy to check that all relations are preserved and that $\mfi$ is an involution.

Finally, let us define two $\mathbb{Q}(\tc,\qc)$-algebra endomorphisms
$$ \gls{a}, \gls{b}: \H^+\to \H^+$$
as follows. For all $i\geq 1$,  
 \begin{align*}
\mfa(T_i)&=T_i, & \mfa(X_i)&=X_i, &\mfa(Y_i)&=\tc^{1-i}(T_{i-1}\cdots T_1)(T_1\cdots T_{i-1})X_iY_i, \\ 
\mfb(T_i)&=T_i, & \mfb(Y_i)&=Y_i,  &\mfb(X_i)&=\tc^{i-1}(T^{-1}_{i-1}\cdots T^{-1}_1)(T^{-1}_1\cdots T^{-1}_{i-1})Y_iX_i.
\end{align*}
The fact that $\mfa$ preserves the defining relations can be directly checked. The only relations that require a small computation are the commutativity relations between the $\mfa(Y_i)$. All the commutativity relations follow, in fact, from 
the commutativity relation between $\mfa(Y_1)$ and $\mfa(Y_2)$. This can be verified as follows
\begin{align*}
T_1X_1 Y_1T_1X_1  Y_1 &= T_1X_1 X_2Y_1T_1 Y_1 & &  \text{by \eqref{XY cross relations}}\\
&= X_1X_2 T_1Y_1T_1 Y_1 & & \text{by \eqref{X relations}}  \\
&=  X_1 X_2 Y_1 T_1Y_1T_1 & &  \text{by \eqref{Y relations}}  \\
&=  X_1 Y_1T_1X_1 Y_1 T_1. & &  \text{by \eqref{XY cross relations}}
\end{align*}
The fact that $\mfa$ is a morphism implies that $\mfb$ is a morphism since 
$$
\mfb=\mfi \mfa\mfi.
$$
Remark that the quadratic relation \eqref{Quadratic} was not used in the verification of the other relations. Therefore, $\mfa$, $\mfb$, and $\mfi$ are not only endomorphisms of $\H^+$, but also endomorphisms of the underlying monoid.

\begin{prop}
The endomorphism $\mfa$ and $\mfb$ of $\H^+$ generate a free monoid.
\end{prop}

\begin{proof}
The double affine Hecke algebra $\H_k$ has endomorphisms defined by the same formulas as $\mfa$, $\mfb$, with the obvious constraint on the label $i$. Indeed, with this definition, $\mfa$ and $\mfb$ correspond, respectively, to 
$\rho_2$ 
and $\rho_1$ in \cite{SV}*{\S2.1}, or to $\mfa$ and $\mfb^{-1}$ in \cite{IS}*{Chapter 6}.  Furthermore, the corresponding $\mfa$, $\mfb$ are in fact automorphisms of $\H_k$ and  generate a copy of the braid group on three 
strands (see, e.g. \cite{IS}*{Theorem 6.4}). They stabilize the subalgebra $\H^+_k$ and the restrictions on $\mfa$, $\mfb$ to $\H^+_k$ generate the free monoid on two generators. Since, as pointed out in Remark \ref{rem: projection}, 
$\H^+_k$ 
is a quotient of the subalgebra  $\H^+(k)\subseteq \H^+$ and $\H^+(k)$ is stable under the action of $\mfa$ and $\mfb$,  we obtain that  the endomorphisms $\mfa$, $\mfb$ of $\H^+$ generate a free monoid.
\end{proof}

\begin{rmk}
We can regard the braid group on three strands $B_3$ as generated by $\mfa, \mfb^{-1}$ satisfying the braid relation 
$$
\mfa\mfb^{-1}\mfa=\mfb^{-1}\mfa\mfb^{-1}.
$$
There is a surjective group morphism $B_3\to \SL(2,\Z)$ which maps $\mfa$ to $A= \begin{bmatrix}1&1\\ 0&1 \end{bmatrix}$ and $\mfb$ to $ B=\begin{bmatrix}1&0\\ 1&1 \end{bmatrix}$. The image of the free monoid $B_3^+$ 
generated by $\mfa$ and $\mfb$ maps bijectively to the free monoid generated by $A$ and $B$ which consists of the set $\SL(2,\Z)^+$ of matrices with non-negative integer entries.
\end{rmk}

Similar facts, pertaining to the algebra $\H^-$, can be recorded with the help of the anti-morphism $\mfe$.
%%%%%%%%%%%%%%%%%%%%%%%%%%%%%%%%%%%%%%%%%%%%%%%%%%%%%%%%%
\subsection{} It is important to remark that the algebras $\H_k^+$, $k\geq 2$, do not form an inverse system in any natural way and therefore $\H^+$ cannot be directly thought off as an inverse limit of $\H_k^+$, $k\geq 2$. However, we 
will show that both $\H^\pm$ acquire natural representations that are constructed by considering inverse systems built from the standard representations of $\H_k^\pm$, $k\geq 2$.

%%%%%%%%%%%%%%%%%%%%%%%%%%%%%%%%%%%%%%%%%%%%%%%%%%%%%%%%%
\section{ The standard representation of the \texorpdfstring{${}^{-}$}~stable  limit DAHA}\label{sec: -stable}

%%%%%%%%%%%%%%%%%%%%%%%%%%%%%%%%%%%%%%%%%%%%%%%%%%%%%%%%%
\subsection{}
The results in the remainder of this section are due to Knop \cite{Kn}. We briefly recall here the  emerging structure. We refer to \S\ref{sec: maps} for the relevant notation.
The map  $\pi_k:\P^-_k\rightarrow \P^-_{k-1}$ is partially compatible with the actions of $\AHA_k^-$ and $\AHA_{k-1}^-$. More precisely, 
\begin{align}\label{neg system}
  \begin{split}
  \pi_k T_i &= T_i \pi_k,\quad 1\leq i\leq k-2,\\
 % \pi_k T_{k-1} &= \pi_k s_{k-1},\\
  \pi_k \widetilde{\omega}_k &= 0, \\
  \pi_k T_{k-1}\widetilde{\omega}_k & =% \tc^{\frac{1}{2}}
  \widetilde{\omega}_{k-1}\pi_k,\\
  \pi_k X_i^{-1} &= X_{i}^{-1}\pi_k,\quad 1\leq i\leq k-1,\\
  \pi_k X_k^{-1} &= 0.
  \end{split}%\tag{3.4.a}
\end{align}
We refer to \cite{Kn}*{Theorem 9.1} for the details.

%%%%%%%%%%%%%%%%%%%%%%%%%%%%%%%%%%%%%%%%%%%%%%%%%%%%%%%%%
\subsection{} \label{subsec: -stable}
The compatibility with the actions of $\H_k^{-}$ and $\H_{k-1}^{-}$ can also be investigated, but the verification is more delicate.  
More precisely, we have the following result \cite{Kn}*{Proposition 9.11}.

\begin{prop}
For any $1\leq i\leq k-1$, we have
\begin{equation}
\pi_k Y_i= Y_{i}\pi_k.
\end{equation}
Furthermore, the operator $Y^{-1}_i$ stabilizes both $\P_k^{-}$ and $\P_{k-1}^-$ and 
\begin{equation}
\pi_k Y^{-1}_i= Y^{-1}_{i}\pi_k.
\end{equation}
\end{prop}

%%%%%%%%%%%%%%%%%%%%%%%%%%%%%%%%%%%%%%%%%%%%%%%%%%%%%%%%%
\subsection{}\label{sec: -limits} For any $n\geq 1$,  the sequence of operators $(A_k)_{k\geq 1}$ defined by 
$$
A_k:=Y_n^{(k)}, \quad k\geq n, 
$$
induces the limit operator $\displaystyle \Y_n:\lim_{\substack{\longleftarrow \\ k\geq n}} \P_k^-\to \lim_{\substack{\longleftarrow \\ k\geq n}} \P_k^-$. Since $\displaystyle \lim_{\substack{\longleftarrow \\ k\geq n}} \P_k$ and $\P^-_\infty$ 
are canonically isomorphic, we obtain an operator $\displaystyle \Y_n:\P^-\to  \P^-$. 
Similarly, we obtain the limit operators $\T_i, \X^{-1}_i$, and $\Y^{-1}_i$, $i\geq 1$. It is important to remark that the operators $\Y^{-1}_i$ are invertible (with inverse $\Y_i$). The following result immediately follows.

\begin{thm}
The limit operators $\T_i, \X^{-1}_i$, and $\Y^{-1}_i$, $i\geq 1$, define a $\H^{-}$ action on $\P^-_\infty$.
\end{thm}
\begin{proof}
All  relations are satisfied because they are satisfied by the corresponding operators acting on each  $\P_k^{-}$. 
\end{proof}

%%%%%%%%%%%%%%%%%%%%%%%%%%%%%%%%%%%%%%%%%%%%%%%%%%%%%%%%%
\subsection{}\label{sec: as} As it was pointed out in \cite{Kn}*{\S 10} the almost symmetric module $\Pas^-$, defined in \S\ref{sec: pas}, is a $\H^-$-stable subspace of $\P_\infty^-$ and is more canonical from a certain point of view. Specifically, each $\P_k^-$ is a parabolic module for the affine Hecke algebra $\AHA_k^-$  and has a standard basis (in the sense of Kazhdan-Lusztig theory) indexed by compositions with at most $k$ parts. The sequences consisting of 
the standard basis elements  indexed by the same composition (in all $\P_k^-$, $k\geq n$, for some $n$) give elements of  $\P_\infty^-$ (see \cite{Kn}*{\S 9}) which are expected to play the role of a standard basis for the limit 
representation. However, these limits of standard basis elements do not span $\P_\infty^-$, but rather the smaller space $\Pas^-$ .

\begin{thm}
The almost symmetric module $\Pas^-$ is a  $\H^{-}$-sub-module of   $\P^-_\infty$.
\end{thm} 
We call this representation the standard representation of $\H^-$.  We expect this representation to be faithful.

As explained in \cite{Kn}, a sequence on non-symmetric Macdonald polynomials indexed by the same composition gives rise to an element of  $\Pas^-$, and such elements are common eigenfunctions for the action the operators 
$\Y^{-1}_i$. 
 The limit  non-symmetric Macdonald polynomials do not span $\Pas^-$ and therefore the spectral theory of the operators $\Y^{-1}_i$ acting on $\Pas^-$ is not yet fully understood.

%%%%%%%%%%%%%%%%%%%%%%%%%%%%%%%%%%%%%%%%%%%%%%%%%%%%%%%%%
\section{ The standard representation of the \texorpdfstring{${}^{+}$}~stable  limit DAHA}\label{sec: +stablelimit}

\subsection{}
Before delving into this section, which is the technical core of the article, it might be helpful to offer a summary of the difficulties that one encounters in the effort to understand the limiting behavior of the Cherednik operators (the analysis of all the other generators of $\H_k^+$ is straightforward) in relationship with the inverse system $(\P_k^+)_{k\geq 2}$, as well as provide a less technical description of the new elements that are employed to describe this limiting behavior.

It has been know to the specialists that the Cherednik operators $Y^{(k)}_i$ are not compatible with the  inverse system $(\P_k^+)_{k\geq 2}$. Our first observation is that the family of operators $\tc^kY^{(k)}_i$, for fixed $i$, are compatible with inverse system $(x_i\P_k^+)_{k\geq 2}$, which does lead to a limit operator acting on $x_i\P_\infty^+$. However, there is no common domain for all limit operators $Y_i$, $i\geq 1$. To address this situation we consider some modified operators acting on $\P_k^+$, which we denote by $\widetilde{Y}_i^{(k)}$. Like the Cherednik operators, these operators satisfy the relations $\widetilde{Y}^{(k)}_{i+1}=\tc^{-1} T_i \widetilde{Y}^{(k)}_i T_i$, $1\leq i\leq k-1$, $T_j \widetilde{Y}^{(k)}_i=\widetilde{Y}^{(k)}_i T_j$, $|i-j|>1$, and are therefore fully determined by the operator  $\widetilde{Y}^{(k)}_{1}$, which has the following additional properties
\begin{enumerate}[label=\roman*)]
\item $\widetilde{Y}_1^{(k)}=\tc^k Y_1^{(k)}$ on $x_1\P_k^+$;
\item $\widetilde{Y}_1^{(k)}\cdot \P_k^+\subseteq x_1\P_k^+$.
\end{enumerate}
The first property implies that $\widetilde{Y}_i^{(k)}=\tc^k Y_i^{(k)}$ on $x_i\P_k^+$ for all $1\leq i \leq k$; the second property implies that $\widetilde{Y}_i^{(k)}\cdot \P_k^+\subseteq T_{i-1}\cdots T_1 x_1\P_k^+$ for all $1\leq i \leq k$. The operator  for $\widetilde{Y}_1^{(k)}$ is obtained by projecting the action of $Y_1^{(k)}$ onto the space $x_1\P_k^+$.

The structural properties of the algebra generated by $\widetilde{Y}_i^{(k)}$, $X_i$, $1\leq k$, and $T_i$, $1\leq i\leq k-1$ are used to define an algebraic structure $\Ht_k^+$ called the deformed DAHA. $\Ht_k^+$ has a natural action $\P_k^+$ (the standard representation), which is related to the standard representation of $\H_k^+$ on $\P_k^+$ in the manner described above.

The modified Cherednik operators $\widetilde{Y}_i^{(k)}$, for fixed $i$, are not yet compatible with the inverse system $(\P_k^+)_{k\geq 2}$, but they are quite close to satisfy such a property. To be able to point out more precisely what happens, we first remark that  they act consistently on \emph{constant sequences} (i.e. sequences compatible with the canonical inclusions $\iota_k: \P_{k-1}^+\to \P_k^+$) in this inverse system (Lemma \ref{Yi consistency}). This fact implies the existence of limit operators defined on $$\displaystyle{\lim_{{\longrightarrow}}} \P_k^+=\Rat(\tc,\qc)[x_1,x_2,\dots],$$ the ring of polynomials in infinitely many variables, but still not on $\P_\infty^+$. Interestingly, the image of $\displaystyle{\lim_{{\longrightarrow}}} \P_k^+$ under the limit operators lies in $\Pas^+$ (which is strictly smaller that $\P_\infty^+$), showing that the smallest space on which one can hope to acquire an action of an algebra of limit operators is $\Pas^+$.

In order to understand the failure of the diagram 
\[
\begin{tikzcd}[row sep=large, column sep=10ex]
  \P_{k}^{+} \arrow[r, "\widetilde{Y}_{1}^{(k)}"] \arrow[d, "\pi_k"'] & \P_{k}^{+} \arrow[d, "\pi_k"] \\
  \P_{k-1}^{+} \arrow[r, "\widetilde{Y}_{1}^{(k-1)}"]                 & \P_{k-1}^{+}                 
\end{tikzcd}
\]
to be commutative, we can consider the difference $\pi_k \widetilde{Y}_{1}^{(k)} - \widetilde{Y}_{1}^{(k-1)} \pi_k: \P_k^+\to \P_{k-1}^+$. On the subspace $x_k\P_k^+$, this difference has no kernel and reduces to 
$\pi_k \widetilde{Y}_{1}^{(k)} : x_k\P_k^+\to \P_{k-1}^+$. Since for every monomial in $m\in \Rat(\tc,\qc)[x_1,x_2,\dots]$ there exists a unique $k$ such that $m\in x_k\P_k^+$, we can use the action of $\pi_k \widetilde{Y}_{1}^{(k)}$ on $x_k\P_k^+$ to define an operator $$W_1: \Rat(\tc,\qc)[x_1,x_2,\dots]\to \Rat(\tc,\qc)[x_1,x_2,\dots],$$
such that $W_1=\pi_k \widetilde{Y}_{1}^{(k)}$ on $x_k\P_k^+$. By restricting $W_1$ to $\P_k^+$ we obtain the operator $W_1^{(k)}: \P_k^+\to \P_k^+$. We can regard the operator $W_1$ as collecting some obvious obstructions to the commutativity of the above diagram.

As it turns out, these are all the obstructions to the commutativity of the diagram: if $ \widetilde{Z}_{1}^{(k)}=  \widetilde{Y}_{1}^{(k)} - W_1^{(k)}$ then the corresponding diagram for the $ \widetilde{Z}_{1}^{(k)} $ operators is commutative, leading to an operator $$ \widetilde{Z}_{1}^{(\infty)}:  \P_\infty^+\to \P_\infty^+. $$

Furthermore, the action of $W_1$ (or $W_1^{(k)}$) on monomials can be explicitly computed; this is the action described in  \S\ref{sec: W}. Since this action is particularly simple, we were able to compute the action of the operators  $W_1^{(k)}$ on sequences  in the inverse system  $(\P_k^+)_{k\geq 2}$. This computation revealed that if the sequence in the inverse system has limit in $\Pas^+$ then one obtains by applying $W_1^{(k)}$ another sequence in $(\P_k^+)_{k\geq 2}$ which can be written as a finite linear expression of sequences compatible with the inverse system. The coefficients in these linear expressions are sequences in $\Rat(\tc,\qc)$ that are convergent in the $\tc$-adic topology (i.e. $\tc$ is treated as a small positive real number). This led us to the concept of limit described in \S\ref{sec: limit}. We refer to the sequences that have limit in this sense as convergent sequences.

 Since both ${W}_i^{(k)}$ and $\widetilde{Z}_i^{(k)}$ have limit in this sense, we obtain that the operators  $\widetilde{Y}_i^{(k)}$ have limit (denoted by $\Y_i$).  Although, the co-domain for the limit operators ${W}_i^{(\infty)}$ and $\widetilde{Z}_i^{(\infty)}$ is, in general, $\P_\infty^+$, it turns out that $\Y_i: \Pas^+\to \Pas^+$. Furthermore, $\Y_i$ satisfies a continuity property: it sends convergent sequences to convergent sequences.

The analysis of the algebraic structure generated by the action of the limit operators $\Y_i$ makes use of the continuity property. Their most remarkable property of the limit operators is their commutativity. The commutativity holds, of course, for the finite rank Cherednik operators, but it does not hold for the modified Cherednik operators and it was interesting to discover that the commutativity is restored in the limit. Furthermore, the action of $\Y_i$ on $x_i\Pas^+$ matches the action of the corresponding limit Cherednik operator $Y_i$. Ultimately, the action of the limit operators $\Y_i$, $\X_i$, and $\T_i$ define an action of $\H^+$ on $\Pas^+$.

\subsection{} \label{sec: +limits}
The map  $\pi_k:\P^+_k\rightarrow \P^+_{k-1}$ is also partially compatible with the actions of $\AHA_k^+$ and $\AHA_{k-1}^+$. More precisely, 

\begin{align}\label{pos system}
  \begin{split}
  \pi_k T_i &= T_i \pi_k,\quad 1\leq i\leq k-2,\\
%    \pi_k T_{k-1} &= \pi_k s_{k-1},\\
  \pi_k T_{k-1}^{-1}\dots T_1^{-1}\widetilde{\omega}_k^{-1} &= 0 \\
  \pi_k \widetilde{\omega}_k^{-1}T_{k-1} & = %\tc^{\frac{1}{2}}
  \widetilde{\omega}_{k-1}^{-1}\pi_k\\
  \pi_k X_i &=    X_{i}\pi_k,\quad  1\leq i\leq k-1, \\
      \pi_k X_i &= 0.
  \end{split}%\tag{3.4.b}
\end{align}
As before, these relations show that the operators $T_i$ and $X_i$ have limits (in the sense of \S \ref{sec: -limits}), denoted by \gls{Tcal} and \gls{Xcal}, that act on $\P_\infty^+$. 

%%%%%%%%%%%%%%%%%%%%%%%%%%%%%%%%%%%%%%%%%%%%%%%%%%%%%%%%%
\subsection{}
On the other hand, for any $i\geq 1$, the actions of  the operators $Y^{(k)}_{i}$ and  $Y^{(k-1)}_{i}$ are no longer compatible with the map $\pi_k$. One immediate obstruction is the fact that the eigenvalues of $Y_{i}$ on $\P_k^+$ and 
$\P_{k-1}^+$ 
no longer match, as is was the case for their action on $\P_k^-$ and $\P_{k-1}^-$. The common spectrum of the action of  $Y_i$, $1\leq i\leq k$, on $\P_k^+$  is known from type $A_{k-1}$ Macdonald theory. The common 
eigenfunctions are non-symmetric Macdonald polynomials; they are indexed by compositions $\lambda=(\lambda_1,\dots,\lambda_k)\in \Z_{\geq 0}^k$. The eigenvalue of $Y^{(k)}_i$ corresponding to the non-symmetric Macdonald 
polynomial indexed by $\lambda$ is 
$$e_\lambda(i):=\qc^{\lambda_i}\tc^{1-w_\lambda(i)},$$
where $w_\lambda(i)=|\{j=1,\dots,i~|~\lambda_j\leq \lambda_i\}|+|\{j=1,\dots,i~|~\lambda_j< \lambda_i\}|$. For the composition  $(\lambda,0)\in \Z_{\geq 0}^{k+1}$, we have 
$$
e_{(\lambda,0)}(i)=e_\lambda(i)\tc^{-1},\quad \text{if } \lambda_i>0, \quad \text{and}\quad e_{(\lambda,0)}(i)=e_\lambda(i),\quad \text{if } \lambda_i=0.
$$
One way to partially correct this problem is to consider the normalized operators $\tc^{k}Y^{(k)}_i$. Comparing the eigenvalues  corresponding to $\lambda$ and $(\lambda,0)$ gives
$$
\tc^{k+1}e_{(\lambda,0)}(i)=\tc^k e_\lambda(i),\quad \text{if } \lambda_i>0, \quad \text{and}\quad \tc^{k+1} e_{(\lambda,0)}(i)=\tc^k e_\lambda(i)\cdot \tc,\quad \text{if } \lambda_i=0.
$$

\begin{ex}
The discrepancy between the eigenvalues is also reflected in the eigenfunctions. For example, the Macdonald polynomial corresponding to the compositions $(0,1,0)$ and $(0,1)$ are, respectively,
$$
x_2+\frac{\qc(1-\tc)}{\qc-\tc^2}x_1\quad\quad \text{and}\quad\quad x_2+\frac{\qc(1-\tc)}{\qc-\tc}x_1.
$$
This shows, in particular, that $\pi_3 Y_1^{(3)}\neq Y_1^{(2)}\pi_3$.
\end{ex}
%%%%%%%%%%%%%%%%%%%%%%%%%%%%%%%%%%%%%%%%%%%%%%%%%%%%%%%%%
\subsection{}

A crucial observation is that the actions of the operators $\tc^{k}Y^{(k)}_iX^{(k)}_i\in \H_k^{+}$ on $\P_{k}^{+}$ are compatible with the inverse system.

\begin{prop}
For any $1\leq i\leq k-1$, we have
$$\pi_k\tc^{k}Y_i X_i=\tc^{k-1}Y_iX_i\pi_k.$$
\end{prop}
\begin{proof}
First note that we have 
\begin{equation}\label{eq1}
\pi_k\omega^{-1}_k T_{k-1}= \omega^{-1}_{k-1}\pi_k,
\end{equation}
which can be verified by direct computation. Hence by  $\eqref{pos system}$ we have
\begin{align*}
\pi_k Y_i^{(k)}X_i^{(k)}&= \tc^{1-i}T_{i-1} \dots T_{1}\pi_k(\omega^{-1}_k T_{k-1}^{-1}\dots T_{i}^{-1}X_i)\\
       		&=\tc^{1-k}T_{i-1} \dots T_{1}\pi_k(\omega^{-1}_k X_kT_{k-1}\dots T_{i})\\
       		&=\tc^{1-k}T_{i-1} \dots T_{1}\omega^{-1}_{k-1}\pi_k( T_{k-1}^{-1}X_kT_{k-1}\dots  T_{i})\\
       		&=\tc^{-k}T_{i-1} \dots T_{1}\omega^{-1}_{k-1}\pi_k( T_{k-1}X_kT_{k-1}\dots T_{i})\\
       		&=\tc^{1-k}T_{i-1} \dots T_{1}\omega^{-1}_{k-1}\pi_k( X_{k-1}T_{k-2}\dots T_{i})\\
       		&=\tc^{-i}T_{i-1} \dots T_{1}\omega^{-1}_{k-1}T_{k-2}^{-1}\dots T_{i}^{-1}X_i\pi_k\\
       		&=\tc^{-1}Y_i^{(k-1)}X_i^{(k-1)}\pi_k
\end{align*}
Therefore we have $\pi_k\tc^{k}Y_i^{(k)}X^{(k)}_i=\tc^{k-1}Y_i^{(k-1)}X^{(k-1)}_i\pi_k$.
\end{proof}

In other words, one can obtain a limit operator corresponding to $Y_i$, acting on the space $x_i\P_\infty^+\subset \P_\infty^+$. However, this is not satisfactory because it does not produce a non-trivial common domain for all the limit operators $Y_i$, $i\geq 1$. In order to extend the action of the limit operator $Y_i$ to a larger domain we introduce the deformed double affine Hecke algebras and the concept of limit detailed in \S \ref{sec: limit}.

%%%%%%%%%%%%%%%%%%%%%%%%%%%%%%%%%%%%%%%%%%%%%%%%%%%%%%%%%
\subsection{The deformed double affine Hecke algebras}\label{sec: dDAHA}

\begin{dfn}
   The deformed DAHA \gls{tHk}, $k\geq 2$, is the $\mathbb{Q}(\tc,\qc)$-algebra generated by the elements $T_1,\dots,T_{k-1}$, $X_1,\dots,X_k$, and \gls{varpi}  satisfying  \eqref{T relation ii}, \eqref{quadratic}, \eqref{X relation ii} 
and the following relations:

        \begin{equation}\label{omega3 relation}
	\begin{gathered}
            \varpi_kT_i=T_{i+1}\varpi_k, \quad 1\leq i\leq k-2,\\
	 \varpi_kX_i=X_{i+1}\varpi_k,\quad 1\leq i\leq k-1,
	\end{gathered}%\tag{5.1.a}
        \end{equation}
        \begin{equation}\label{gamma relation}
	\begin{gathered}
            \gamma_k T_{k-1}=-\tc \gamma_k,\quad T_1\gamma_k =\gamma_k,\\
	 \gamma_k\varpi_k^{k-2}\gamma_k=\gamma_k\varpi_k^{k-1}\gamma_k=\gamma_k\varpi_k^{k}=0,
	\end{gathered}%\tag{5.1.b}
        \end{equation}
 where        $$\gamma_k=\varpi_k^2 T_{k-1}-T_{1}\varpi_k^2.$$
\end{dfn}

\begin{rmk}
The quotient of $\Ht_k^+$ by the ideal generated by \gls{gamma} is isomorphic to $\H_k^+$. As it is clear from the relations \eqref{gamma relation}, $\gamma_k$ is not central in $\Ht_k^+$.
\end{rmk}

\begin{note}
Let \gls{tY}, $1\leq i\leq k$, be defined as follows
$$\widetilde{Y}_1=\tc^{k}\varpi_kT_{k-1}^{-1}\dots T_1^{-1},\qquad \widetilde{Y}_{i+1}=\tc^{-1} T_i \widetilde{Y}_i T_i, \quad 1\leq i\leq k-1.$$
Therefore,
$$
\widetilde{Y}_i =  \tc^{1-i+k}T_{i-1} \dots T_{1}\varpi_k T_{k-1}^{-1}\dots T_{i}^{-1}, \quad 1\leq i\leq k.
$$
\end{note}
\begin{rmk}\label{rem: crossrel} It is important to remark that the relation \eqref{XY cross relation} is also satisfied in $\Ht_k$. Indeed
\begin{align*}
\widetilde{Y}_{1} T_1 X_1 &=\tc^{k}\varpi_kT_{k-1}^{-1}...T_2^{-1}X_1 \\
	&=\tc^{k}X_2\varpi_kT_{k-1}^{-1}...T_2^{-1} \\
	&=X_2 \widetilde{Y}_{1} T_1.
\end{align*}
The relations
$$
T_j \widetilde{Y}_i=\widetilde{Y}_i T_j, \quad |i-j|>1,
$$
are also satisfied. The verification is a simple consequence of the braid relations and \eqref{omega3 relation}.
\end{rmk}
\begin{rmk}\label{rem: [y,y]}
 On the other hand, the elements $\widetilde{Y}_i$, $1\leq i\leq k$, do not commute. In fact, we have
$$[\widetilde{Y}_1, \widetilde{Y}_2]=\tc^{2k-1}\gamma_k T_{k-1}^{-1}\dots T_1^{-1}T_{k-1}^{-1}\dots T_2^{-1}=\tc^{2k-1}\gamma_k T_{k-2}^{-1}\dots T_1^{-1}T_{k-1}^{-1}\dots T_1^{-1}.$$
Therefore, $\gamma_k$ can be seen as an obstruction to the commutativity of the elements $\widetilde{Y}_i$.
\end{rmk}

Since the notation for the generators of $\Ht_k^+$  and the elements $\widetilde{Y}_i$ does not specify the integer $k$, we will use the notation in the fashion indicated in Convention \ref{convention}.

%%%%%%%%%%%%%%%%%%%%%%%%%%%%%%%%%%%%%%%%%%%%%%%%%%%%%%%%%
\subsection{} \label{sec: dDAHArep}

For any $1\leq i \leq k$, let $\pr_i:\P_k^+ \to \P_k^+$, be the $\mathbb{Q}(\tc,\qc)$-linear map which acts as identity on monomials divisible by $x_i$ and as the zero map on monomials not divisible by $x_i$. In other words, \gls{pr}  is the 
projection onto the subspace $x_i\P_k^+$. 

We  define the following action of $\Ht_{k}^{+}$ on $\P_{k}^{+}$. The action is related to the action of $\H_k$ on $\P_k^+$.

\begin{thm}
   The following formulas define an action of $\Ht_{k}^{+}$ on $\P_k^+$:

  \begin{align}\label{modified DAHA representation}
      \begin{split}
	    T_i f(x_1,\dots,x_k) &= s_i f(x_1,\dots,x_k) +(1-\tc)x_i\frac{1-s_i}{x_i-x_{i+1}}f(x_1,\dots,x_k), \\
	    X_i f(x_1,\dots,x_k) &=x_i f(x_1,\dots,x_k),\\
	    \varpi_k f(x_1,\dots,x_k) &=\pr_1f(x_2,\dots,x_k, \qc x_1).
	          \end{split}%\tag{5.2}
  \end{align}
 
\end{thm}
Note that the action of $ \varpi_k $ coincides with the action of $\pr_1\omega_k^{-1}$.
\begin{proof}
We only need to check the relations \eqref{omega3 relation} and \eqref{gamma relation}. From \eqref{omega2 relation} we have 
\begin{align*}
\varpi_kT_i f&=\pr_1\omega_k^{-1}T_i f\\
       		&=\pr_1(T_{i+1}\omega_k^{-1} f)\\
       		&=T_{i+1}\pr_1\omega_k^{-1} f\\
       		&=T_{i+1}\varpi_k f,
\end{align*}
for $f\in \P_{k}^{+}$ and $1\leq i\leq k-2$. Similarly, from \eqref{X-omega2 cross relation} we have 
$$\varpi_k X_i f=X_{i+1}\varpi_k f,$$
for $f\in \P_{k}^{+}$ and $1\leq i\leq k-1$.

By linearity it suffices to verify the relation \eqref{gamma relation} on monomials. A straightforward computation leads to the following explicit expression of the action of $\gamma_k$

\begin{equation}\label{eq: gamma}
\gamma_k(x_1^{t_1}\dots x_{k-1}^{t_{k-1}}x_{k}^{t_{k}}) = \begin{cases} 
      0, & \textrm{  if  } t_{k-1}\neq 0,t_k\neq 0\\\\
      (1-\tc)\qc^{t_{k-1}}(x_1^{t_{k-1}-1}x_2+x_1^{t_{k-1}-2}x_2^2+\dots \\+x_1 x_2^{t_{k-1}-1})(x_3^{t_1}x_4^{t_2}\dots x_k^{t_{k-2}}), & \textrm{  if  } t_{k-1}\neq 0,t_k= 0\\\\
      (\tc-1)\qc^{t_{k}}(x_1^{t_{k}-1}x_2+x_1^{t_{k}-2}x_2^2+\dots\\+x_1 x_2^{t_{k}-1})(x_3^{t_1}x_4^{t_2}\dots x_k^{t_{k-2}}), &\textrm{  if  } t_{k-1}= 0,t_k\neq 0
   \end{cases}
\end{equation}
Note that the expression immediately implies that the actions of $\gamma_k T_{k-1}$ and $-\tc \gamma_k$ coincide, and similarly for the actions of $T_1\gamma_k$ and $\gamma_k$. The remaining relations can also be directly verified.
\end{proof}

\begin{rmk}\label{rem: modCherednik}
The action of $\widetilde{Y}_i X_i$  coincides with the action of $\tc^{k}Y_i X_i\in \H_k^+$ on $\P_k^+$. Equivalently, the action of $\tc^{-k}\widetilde{Y}_i$ and the action of $Y_i\in \H_k^+$ coincide on $x_i\P_k^+$. Therefore $\widetilde{Y}_i$ can be 
regarded as a modification of the Cherednik operator $Y_i$.
\end{rmk}

%%%%%%%%%%%%%%%%%%%%%%%%%%%%%%%%%%%%%%%%%%%%%%%%%%%%%%%%%
\subsection{} \label{sec: ind}
We are now ready to explore the compatibility between the action of the operators \gls{tY(k)} and the map $\pi_k: \P_k^+\to \P_{k-1}^+$. 
\begin{lem} \label{Yi consistency}
 The following diagram is commutative:
\[
\begin{tikzcd}
  \P_{k-1}^{+} \arrow[r, "\widetilde{Y}_i^{(k-1)}"] \arrow[d, "\iota_k"'] & \P_{k-1}^{+} \\
  \P_{k}^{+} \arrow[r, "\widetilde{Y}_i^{(k)}"]                 & \P_{k}^{+}               \arrow[u, "\pi_k"]   
\end{tikzcd}
\]
Equivalently, we have
  $$\pi_{k}\widetilde{Y}_i^{(k)} f=\widetilde{Y}_i^{(k-1)} f,$$
  for all $f\in \P_{k-1}^{+}$. 
\end{lem}
\begin{proof}
Note that $\pr_1\omega^{-1}_{k}=\omega^{-1}_{k}\pr_k$. For any $f\in \P_{k-1}^{+}$ we have
\begin{align*}
\pi_{k}\widetilde{Y}_i^{(k)}\iota_k f&=\pi_{k}\tc^{1-i+k}T_{i-1} \dots T_{1} \pr_1\omega^{-1}_{k}T_{k-1}^{-1}\dots T_i^{-1}f\\
       		&=\tc^{k-i}T_{i-1} \dots T_{1}\pr_1\pi_{k}\omega^{-1}_{k}T_{k-1}T_{k-2}^{-1}\dots T_i^{-1}f+\pi_{k}\tc^{k-i}(\tc-1)T_{i-1} \dots T_{1}\pr_1\omega^{-1}_{k}T_{k-2}^{-1}\dots T_i^{-1}f\\
       		&=\tc^{k-i}T_{i-1} \dots T_{1}\pr_1\omega^{-1}_{k-1}\pi_{k}T_{k-2}^{-1}\dots T_i^{-1}f+\pi_{k}\tc^{k-i}(\tc-1)T_{i-1} \dots T_{1}\omega^{-1}_{k}\pr_k(T_{k-2}^{-1}\dots T_i^{-1}f)\\
       		&=\widetilde{Y}_i^{(k-1)}\pi_{k}f+0
\end{align*}
For the third equality we used \eqref{eq1}.
\end{proof}

As an immediate consequence we obtain the following.

\begin{prop}
For any $i\geq 1$, the sequence of operators $(\widetilde{Y}_i^{(k)})_{k\geq 2}$ induces a map
$$
\gls{tYinf}:  \lim_{{\longrightarrow}} \P_k^+\to \P_{\infty}^{+},
$$
such that $\Pi_k \widetilde{Y}_i^{(\infty)} I_k=\widetilde{Y}_i^{(k)}$, for all $k\geq i$, $k\geq 2$.
\end{prop}
\begin{proof}
Fix $n\geq i$, and denote by $\iota_{n,k}: \P_n^+\to \P_k^+$, $n\leq k$, the canonical inclusion. By Lemma \ref{Yi consistency}, the sequence of maps $\widetilde{Y}_i^{(k)} \iota_{n,k}: \P_n^+\to \P_k^+$, $k\geq n,i$ is compatible with the 
structure maps $\pi_k$. Therefore, they induce a morphism
$$
\P_n^+\to  \lim_{\substack{\longrightarrow \\ k\geq n}} \P_k^+\cong \displaystyle \lim_{{\longrightarrow }} \P_k^+.
$$
Furthermore, Lemma \ref{Yi consistency} implies that these maps are compatible with the structure maps $\iota_n$ and therefore induce the desired morphism.
\end{proof}

\begin{ex}\label{exp: x_2^2}
 To illustrate the construction of the limit operator $\widetilde{Y}_1^{(\infty)}: \displaystyle{\lim_{{\longrightarrow}}} \P_k^+\to \P_{\infty}^{+}$ let us compute its action on the element
$x_2^2\in \displaystyle{\lim_{{\longrightarrow}}} \P_k^+$. A direct computation gives

\begin{align*}
\widetilde{Y}_1^{(2)}\cdot x_2^2&=\qc^2 \tc(\tc-1)x_1^2+\qc\tc(\tc-1)x_1 x_2\\
\widetilde{Y}_1^{(3)}\cdot  x_2^2&=\qc^2 \tc(\tc-1)x_1^2+\qc\tc(\tc-1)x_1 x_2-\qc\tc(\tc-1)^2 x_1x_3\\
\widetilde{Y}_1^{(4)}\cdot  x_2^2&=\qc^2 \tc(\tc-1)x_1^2+\qc\tc(\tc-1)x_1 x_2-\qc\tc(\tc-1)^2 (x_1x_3+x_1x_4)
\end{align*}
and, more generally, for $k\geq 4$,
$$ \widetilde{Y}_1^{(k)}\cdot  x_2^2=\qc^2 \tc(\tc-1)x_1^2+\qc\tc(\tc-1)x_1 x_2-\qc\tc(\tc-1)^2 (x_1x_3+\cdots x_1x_k).$$
Therefore, the limit operator acts as
$$\widetilde{Y}_1^{(\infty)}\cdot  x_2^2=\qc^2 \tc(\tc-1)x_1^2+\qc\tc(\tc-1)x_1 x_2-\qc\tc(\tc-1)^2 x_1 e_1[\Xb_2].$$

\end{ex}

On the other hand, the action same operator on a non-constant sequence (e.g. one that converges to a non-trivial almost symmetric function) shows a slightly different behaviour.
\begin{ex} \label{exp: p_2}
 Let $F[\Xb]=x_2^2+x_3^2+\cdots=p_2[\Xb_1]$, which is the inverse limit of the sequence  $$F_{k}=\Pi_k F[\Xb]=x_2^2+\cdots+x_k^2,~ k\geq 2.$$ We have,
\begin{align*}
 \widetilde{Y}_1^{(2)}\cdot  F_{2}&=\qc^2(\tc^2-\tc)x_1^2+\qc\tc(\tc-1)x_1 x_2\\
\widetilde{Y}_1^{(3)}\cdot  F_{3}&=\qc^2(\tc^3-\tc)x_1^2+\qc\tc(\tc-1)(x_1 x_2+x_1 x_3)
\end{align*}
and, more generally, for $k\geq 2$,
$$ \widetilde{Y}_1^{(k)}\cdot  F_{k}=\qc^2(\tc^k-\tc)x_1^2+\qc\tc(\tc-1)(x_1 x_2+x_1 x_3+\cdots+x_1x_k).$$
As it can be clearly seen, the sequence $ \widetilde{Y}_1^{(k)}F_{k}$ fails to have an inverse limit because of the contribution of the term $\qc^2(\tc^k-\tc)x_1^2$. 
\end{ex}
As we will see next, the general behavior of the action of fixed operator on a sequence with inverse limit is no more complicated than the one exhibited in Example \ref{exp: p_2}. The operators $W_i^{(k)}$ introduced in the next section formalize this observation. 

%%%%%%%%%%%%%%%%%%%%%%%%%%%%%%%%%%%%%%%%%%%%%%%%%%%%%%%%%
\subsection{} \label{sec: W}
We define the $\mathbb{Q}(\tc,\qc)$-linear operators $\gls{W(k)}:   \P_k^+\to \P_{k}^{+}$, $i\geq 1$, as follows. Let $s\geq 0$,  and let $i_1<i_2<\dots <i_s\leq k$ and $t_1, \dots, t_s$ be positive integers. We 
set 
\begin{equation}
W^{(k)}_1(x_{i_1}^{t_{1}}\dots x_{i_s}^{t_{s}})=\begin{cases}
      0, & \text{  if  } 1=i_1 \text{ or  }s=0,\\\\\
      \tc^{i_s}(1-\tc^{-1})\qc^{t_{s}}x_{1}^{t_{s}}x_{i_1}^{t_{1}}\dots x_{i_{s-1}}^{t_{s-1}}, & \text{  if  } i_1>1.
\end{cases}
\end{equation}
Define
$$W^{(k)}_{i+1}=\tc^{-1} T_i W^{(k)}_i T_i,\quad  1\leq i\leq k-1.$$
For $1\leq i\leq k$, we also denote
$$\gls{tZ}=\widetilde{Y}_{i}^{(k)}-W_{i}^{(k)}.$$

\begin{lem}
The following diagram is commutative:
\[
\begin{tikzcd}[row sep=large, column sep=10ex]
  \P_{k}^{+} \arrow[r, "\widetilde{Z}_{i}^{(k)}"] \arrow[d, "\pi_k"'] & \P_{k}^{+} \arrow[d, "\pi_k"] \\
  \P_{k-1}^{+} \arrow[r, "\widetilde{Z}_{i}^{(k-1)}"]                 & \P_{k-1}^{+}                 
\end{tikzcd}
\]
\end{lem}

\begin{proof}
It is enough to prove the statement for $i=1$.  For any monomial $f\in \P_{k-1}^{+}$, we have
$$\pi_k W_1^{(k)} f =W_1^{(k-1)}\pi_k f$$
and the commutativity of the digram follows from Lemma \ref{Yi consistency}. 

Let $f\in  x_k \P_{k}^{+}$ be a monomial and $g\in \P_{k}^{+}$. A direct check shows that
$$
\pi_k\omega^{-1}_{k}\pr_kT_{k-1}^{-1}\pr_k g \neq 0
$$
only if $g$ is not divisible by $x_{k-1}$. On the other hand, 
\begin{equation*}
T_i^{-1} x_i\P_k^+ \subseteq x_{i+1}\P_k^+.
\end{equation*}
Now, 
\begin{align*}
\pi_k \widetilde{Y}_{1}^{(k)}f&= \tc^k \pi_k\varpi_k T_{k-1}^{-1}\dots T_{1}^{-1} f\\
&=\tc^k \pi_k\varpi_k T_{k-1}^{-1}\dots T_{1}^{-1}\pr_k f\\
&=\tc^k \pi_k\omega^{-1}_{k}\pr_kT_{k-1}^{-1}\pr_k (T_{k-2}^{-1}\dots T_{1}^{-1} f).
\end{align*}
Based on the previous remarks, $\pi_k \widetilde{Y}_{1}^{(k)}f\neq 0$ unless $f$ is not divisible by $x_1$. Furthermore, the only monomial from $T_1^{-1}f$ that survives is $s_1 f$. Applying this repeatedly, we obtain 
$$
\pi_k \widetilde{Y}_{1}^{(k)}f=\tc^k \pi_k\omega^{-1}_{k}\pr_kT_{k-1}^{-1}\pr_k (s_{k-2}\dots s_1 f)=W_{1}^{(k)}f.
$$
Therefore, $\pi_k\widetilde{Z}_{i}^{(k)} f=0=\widetilde{Z}_{i}^{(k-1)} \pi_k f$, as expected.
\end{proof}

As a consequence, we obtain the following.

\begin{prop}\label{prop: Z}
For any $i\geq 1$, the sequence of operators $(\widetilde{Z}_i^{(k)})_{k\geq 2}$ induces a map
$$
\gls{tZinf}:  \P_\infty^+\to \P_{\infty}^{+},
$$
such that $\Pi_k \widetilde{Z}_i^{(\infty)}=\widetilde{Z}_i^{(k)}$, for all $k\geq i$, $k\geq 2$.
\end{prop}

\begin{ex} \label{exp: x_2^2ii}
In continuation of Example \ref{exp: x_2^2} we illustrate the actions of the operators $W_1^{(k)}$ and $\widetilde{Z}_1^{(k)}$ on the constant sequence $x_2^2$. We have,
$$W_1^{(k)}\cdot  x_2^2=\qc^2 \tc(\tc-1)x_1^2 \quad \text{ and }\quad \widetilde{Z}_1^{(k)}\cdot  x_2^2=\qc\tc(\tc-1)x_1 x_2-\qc\tc(\tc-1)^2 (x_1x_3+\cdots +x_1x_k),~  k\geq 3.$$
Clearly, the both sequence have inverse limit, in particular,
$$\widetilde{Z}_1^{(\infty)}\cdot  x_2^2=\qc\tc(\tc-1)x_1 x_2-\qc\tc(\tc-1)^2 x_1 e_1[\Xb_2].$$
\end{ex}

\begin{ex} \label{exp: p_2ii}
In continuation of Example \ref{exp: p_2} we illustrate the actions of the operators $W_1^{(k)}$ and $Z_1^{(k)}$ on the sequence $F_k=x^2_2+\cdots +x^2_k$. We have,
\begin{align*}
W_1^{(2)}\cdot  F_{2}&=W_1^{(2)}x_2^2=\qc^2(\tc^2-\tc)x_1^2,\\
W_1^{(3)}\cdot  F_{3}&=W_1^{(3)}(x_2^2+x_3^2)=\qc^2(\tc^2-\tc)x_1^2+\qc^2(\tc^3-\tc^2)x_1^2=\qc^2(\tc^3-\tc)x_1^2,
\end{align*}
and, more generally, for $k\geq 2$,
$$W_1^{(k)}\cdot  F_{k}=\qc^2(\tc^k-\tc)x_1^2.$$
Furthermore, 
$$\widetilde{Z}_1^{(k)} \cdot F_{k}=\widetilde{Y}_1^{(k)}\cdot  F_{k}-W_1^{(k)}\cdot  F_{k}=\qc\tc(\tc-1)(x_1 x_2+x_1 x_3+...+x_1x_k),$$
which can be seen to have inverse limit. Therefore,
$$\widetilde{Z}_1^{(\infty)}\cdot F[\Xb]=\qc\tc(\tc-1)x_1 e_1[\Xb_1].$$
\end{ex}

%%%%%%%%%%%%%%%%%%%%%%%%%%%%%%%%%%%%%%%%%%%%%%%%%%%%%%%%%
\subsection{} \label{sec: limit}  Proposition \ref{prop: Z}  motivates the concept of limit we define as follows; we emphasize that this concept of limit depends intrinsically on the structure of the subspace $\Pas^+\subset \P^+_\infty$.

Let $R(\tc,\qc)=A(\tc,\qc)/B(\tc,\qc)\in \mathbb{Q}(\tc,\qc)$, with $A(\tc,\qc),~B(\tc,\qc)\in \mathbb{Q}[\tc,\qc]$. The order of vanishing at $\tc=0$ for  $R(\tc,\qc)$,
 denoted by
$$
\ord R(\tc,\qc),
$$
 is the difference between the order of vanishing at $\tc=0$ for $A(\tc,\qc)$ and $B(\tc,\qc)$.

We say that the sequence $(a_n)_{n\geq 1}\subset \mathbb{Q}(\tc,\qc)$ converges to $0$ if the sequence $(\ord a_n)_{n\geq 1}\subset \Z$ converges to $+\infty$. We say that the sequence $(a_n)_{n\geq 1}\subset \mathbb{Q}(\tc,\qc)$ converges 
to $a$ if $(a_n-a)_{n\geq 1}$ converges to $0$. We write,
$$
\lim_{n\to \infty} a_n=a.
$$

\begin{dfn}\label{Def of limit}
Let $(f_k)_{k\geq 1}$ be a sequence with $f_k\in \P^+_k$. We say that the sequence is convergent if there exists $N\geq 1$ and sequences $(h_k)_{k\geq 1}$, $(g_{i,k})_{k\geq 1}$, $i\leq N$, $h_k, ~g_{i, k}\in \P^+_k$, and $(a_{i, 
k})_{k\leq 1}$, $i\leq N$, $a_{i, k}\in \mathbb{Q}(\tc,\qc)$ such that
\begin{enumerate}[label={\alph*)}]
\item For any $k\geq 1$, we have $f_k=h_k+\sum_{i=1}^N a_{i, k} g_{i, k}$;
\item For any $i\leq N$, $k\geq 2$,  $\pi_k(g_{i, k})=g_{i, k-1}$ and $\pi_k(h_{k})=h_{k-1}$. We denote by $$\displaystyle g_i=\lim_{k\to \infty} g_{i,k}\quad  \text{and}\quad  \displaystyle h=\lim_{k\to \infty} h_{k}$$ the sequence 
$(g_{i,k})_{k\geq 
1}$ and, respectively, $(h_{k})_{k\geq 1}$ as elements of $\P_\infty^+$.  We require that $g_i\in \Pas^+$.
\item For any  $i\leq N$ the sequence $(a_{i, k})_{k\geq 1}$ is convergent. We denote  $\displaystyle a_i=\lim_{k\to \infty}(a_{i, k})$.
\end{enumerate}
If the sequence $(f_k)_{k\geq 1}$ is convergent we define its limit as $$\lim_k(f_k):=h+ \sum_{i=1}^N a_i g_i\in\P_\infty^+.$$
\end{dfn}

\begin{ex} The sequence 
$$f_k=(1+\tc+...+\tc^k)e_i[\overline{\mathbf{X}}_{k}],$$
 has the limit
$$\lim_k f_k=\frac{1}{1-\tc}e_i[\Xb].$$
The sequence  
$$g_k=\tc^k e_i[\overline{\mathbf{X}}_{k}]$$
has  limit $0$.
\end{ex}

We show that the limit of a sequence does not depend on the choice of the auxiliary sequences in Definition \ref{Def of limit}. 

\begin{prop}\label{prop: limit}
The concept of limit is well-defined.
\end{prop}
\begin{proof}
It suffices to show that the limit of the constant sequence $0$ is zero, regardless of the auxiliary sequences in Definition \ref{Def of limit}. Consider sequences $(c_{i,k})_{k\geq 1}$ and $(q_{i,k})_{k\geq 1}$ such that 
$$0=\sum_{i=1}^{N}c_{i,k}q_{i,k}\in \P_{k}^{+},$$
and
$$\lim_{k\rightarrow\infty}c_{i,k}=c_i\in \mathbb{Q}(\tc,\qc),\ \lim_{k\rightarrow\infty}q_{i,k}=q_i\in\Pas^{+}.$$
We need to show  that $\sum_{i=1}^N c_iq_i=0$. 

Without loss of generality, we assume that $q_1,...,q_N$ are $\mathbb{Q}$-linearly independent.  Indeed, any linear relation between $q_1,...,q_N$ must also hold for  $q_{1,k},...,q_{N,k}$, for all $k$. We can therefore substitute one of 
them, say  $q_{1,k}$, with the same $\mathbb{Q}$-linear combination of $q_{2,k},...,q_{N,k}$ for all $k$. It is clear that the conclusion does not change after such a substitution.

Recall that each $q_i\in\Pas^{+}$. We can find $n\geq 1$, such that $q_i\in\P(n)$ for all $1\leq i\leq N$. For the same reason as before, without loss of generality, we can assume that
$$q_i= f_i(x_1,...,x_n)e_{\alpha_i}[\mathbf{X}_{n}]$$
where
$$\alpha_i=(\alpha_{i,1}\geq\dots \geq\alpha_{i,s_i}), \quad 1\leq i\leq N$$ are distinct partitions. Let
$$M=\max_{1\leq i\leq N}\alpha_{i,1}.$$
For any $k>M+n$, we claim that $q_{1,k},...,q_{N,k}$ are also $\mathbb{Q}$-linear independent. Indeed, if we have a linear relation
$$\sum_{i=1}^{N}a_i f_i(x_1,...,x_n)e_{\alpha_i}[x_{n+1},...,x_{k}]=0,$$
for some $a_i\in \mathbb{Q}$, then for any evaluation at $x_1=b_1,...,x_n=b_n,$ we have
$$\sum_{i=1}^{N}a_i f_i(b_1,...,b_n)e_{\alpha_i}[x_{n+1},...,x_{k}]=0.$$
Note that, because $k-n>M$, and the partitions $\alpha_i$ are distinct,  the symmetric functions $e_{\alpha_i}[\mathbf{X}_{n}]$ are linearly independent. Therefore, for all $1\leq i\leq N$,
$$a_i f_i(b_1,...,b_n)=0.$$
Since $f_i$ are polynomials in finitely many variables, we obtain,  for all $1\leq i\leq N$,
$$a_i f_i(x_1,...,x_n)=0,$$
which is a contradiction. Therefore, for $k$ large enough, we have that $q_{1,k},...,q_{N,k}$ are $\mathbb{Q}$-linear independent.

We can now prove that for $k$ large enough, all $c_{i,k}$ will necessarily be $0$. Indeed, from the hypothesis we have
$$0=\sum_{i=1}^{N}c_{i,k}q_{i,k}.$$
By multiplying both sides the common denominator of $c_{i,k}$ we may assume all $c_{i,k}$ are polynomials in two variables $\tc,\qc$. Again, for any evaluation $\qc=a,\tc=b$, we have
$$0=\sum_{i=1}^{N}c_{i,k}(a,b)q_{i,k}.$$
But we already now $q_{1,k},...,q_{N,k}$ are $\mathbb{Q}$-linear independent for $k$ large enough. Hence it forces $c_{i,k}(a,b)=0$ for all $i$. Again, this implies that $c_{i,k}(\tc,\qc)=0$, $1\leq i\leq N$.
\end{proof}

%%%%%%%%%%%%%%%%%%%%%%%%%%%%%%%%%%%%%%%%%%%%%%%%%%%%%%%%%
\subsection{} \label{sec: cont}
For later use, we record the following result. To set the notation, assume that $A_k:\P_k^+\to \P_k^+$, $k\geq 1$, is a sequence of operators with the following property: for any $f\in \Pas^+$, the sequence
$(A_k \Pi_k f)_{k\geq 1}$ converges to an element of $\Pas^+$. Let $A$ be the operator 
$$
A: \Pas^+\to \Pas^+,\quad f\mapsto \lim_k A_k \Pi_k f.
$$
We refer to $A$ as the limit operator of the sequence $(A_k)_{k\geq 1}$.
\begin{prop}\label{prop: continuity}
Let $(f_k)_{k\geq 1}$, $f_k\in \P_k^+$ be a convergent sequence such that $f=\lim_{k} f_k\in \Pas^+$. Then, with the notation above, we have 
$$A f=\lim_{k}A_{k} f_k.$$
\end{prop}

\begin{proof}
By replacing $f_k$ with $f_k-\Pi_k$, it is enough to prove the statement for the case $f=0$.
Assume that $f=0$. In this case, 
there exist sequences $(c_{i,k})_{k\geq 1}$ and $(q_{i,k})_{k\geq 1}$ as in Definition \ref{Def of limit} such that 
$$f_k=\sum_{i=1}^{N}c_{i,k}q_{i,k}\in \P_{k}^{+},$$
with 
$$\lim_{k\rightarrow\infty}c_{i,k}=c_i\in \mathbb{Q}(\tc,\qc),\ \lim_{k\rightarrow\infty}q_{i,k}=q_i\in\Pas^{+}.$$
For each $i$ such that $c_i\neq 0$, we may replace $c_{i,k}$ with $c_{i,k}/c_i$ and $q_{i,k}$ with $c_i q_{i,k}$. This allows us to assume that $c_i\in \{0,1\}$.
Without loss of generality, we assume that $q_1,...,q_N$ are $\mathbb{Q}$-linearly independent.  Indeed, any linear relation between $q_1,...,q_N$ must also hold for  $q_{1,k},...,q_{N,k}$, for all $k$. We can therefore substitute one of 
them, say  $q_{1,k}$, with the same $\mathbb{Q}$-linear combination of $q_{2,k},...,q_{N,k}$ for all $k$. It is clear that the conclusion does not change after such a substitution. It is now clear that for all $i$ we have $c_i=0$.

For any $k\geq 1$, 
$$
A_k f_k=\sum_{i=1}^N c_{i,k} A_k q_{i,k}=\sum_{i=1}^N c_{i,k} A_k \Pi_k q_i.
$$
Since  $\displaystyle \lim_{k\rightarrow\infty}A_k \Pi_k q_i=A q_i\in\Pas^{+}$, we have, by Definition \ref{Def of limit},
$$
\lim_k  A_k f_k=\sum_{i=1}^N c_{i,k} A q_i=0.
$$ 
This is precisely our claim.
\end{proof}
This result can be interpreted as a property of continuity for the operator $A$. Let $(B_k)_{k\geq 1}$ be another sequence of operators with the same property as $(A_k)_{k\geq1}$ and denote by $B: \Pas^+\to \Pas^+$ the 
corresponding limit operator.
\begin{cor}\label{cor: AB}
With the notation above, the operator $AB$ is the limit of the sequence $(A_kB_k)_{k\geq 1}$.
\end{cor}
\begin{proof}
Let $f\in \Pas^+$. Since $\displaystyle \lim_k B_k\Pi_kf=B f\in \Pas^+$, we can apply Proposition \ref{prop: continuity} to obtain
$$
\lim_k A_k B_k \Pi_k f=A B f \in \Pas^+,
$$
which proves our claim.
\end{proof}

%%%%%%%%%%%%%%%%%%%%%%%%%%%%%%%%%%%%%%%%%%%%%%%%%%%%%%%%%
\subsection{}\label{sec: Wlim}
Let us examine the following situation.

\begin{lem}\label{lemma: W}
Let $f\in \Pas^+$ and $i\geq 1$. The sequence $W_i^{(k)} \Pi_k f$ converges. We define
$$\gls{Winf}: \Pas^+\to \P_\infty^+,\quad W^{(\infty)}_i f=\lim_{k} W_i^{(k)}\Pi_k f\in \P_\infty^{+}.$$
\end{lem}
\begin{proof}
It is enough to prove our claim for $W_1$ and $f=g(x_1,\dots,x_n)m_{\alpha}[\mathbf{X}_n]$, for some $n\geq 1$, where $g$ is a monomial and $m_\alpha$ is the monomial symmetric function in the indicated alphabet. If $x_1$ divides 
$g$, then $W_1f=0$. We assume that $x_1$ does not divide $g$.

We denote by $\ell(\alpha)$ the length of the partition $\alpha$. Of course, for any $k\geq n+\ell(\alpha)$, we have
$$\Pi_k m_{\alpha}[\mathbf{X}_n]=m_{\alpha}[x_{n+1},\dots ,x_k].$$

%For $\ell(\alpha)=1$ and $k\geq n+1$, we have
%\begin{align*}
%W_1^{(k)} \Pi_k f&=W_1^{(k)} g(x_1,\dots,x_n) m_{\alpha_1}[x_{n+1},\dots ,x_k]\\
% &=\qc^{\alpha_1}(\tc^{k}-\tc^n)x_1^{\alpha_1}g(x_1,\dots,x_n),
%\end{align*}
%and  $\lim W_1^{(k)}f=-\qc^{\alpha_1}\tc^n x_1^{\alpha_1}g(x_1,\dots,x_n)$.

If $k\geq n+\ell(\alpha)$, denote by $a_1,\dots, a_s$ the distinct parts of $\alpha$. Let $\beta_i$, $1\leq i\leq s$ the partition obtained by eliminating one part of size $a_i$ from $\alpha$. Then,
$$
m_\alpha[x_{n+1},\dots,x_k]=\sum_{i=1}^s \sum_{j=n+\ell(\alpha)}^k  m_{\beta_i}[x_{n+1},\dots, x_{j-1}] x_j^{a_i}.
$$

 Therefore, we have
\begin{align*}
W_{1}^{(k)}\Pi_k f &=g(x_1,\dots,x_n) \sum_{i=1}^s \sum_{j=n+\ell(\alpha)}^k W_1^{(k)}  m_{\beta_i}[x_{n+1},\dots, x_{j-1}] x_j^{a_i}\\
&=g(x_1,\dots,x_n) \sum_{i=1}^s \sum_{j=n+\ell(\alpha)}^k (\tc^j-\tc^{j-1})\qc^{a_i} x_1^{a_i}  m_{\beta_i}[x_{n+1},\dots, x_{j-1}].
\end{align*}
For a monomial $m=x_{i_1}^{\eta_1}\cdots x_{i_M}^{\eta_M}$, $n<i_1<\dots<i_M$, $\eta_1,\dots,\eta_M\geq 1$, we denote $\ell(m)=i_M$ and we write $m\in [\beta]$ if $(\eta_1,\dots, \eta_M)$ is a permutation of $\beta$ 
and we write  $m\in [\beta]_k$ if  $m\in [\beta]$ and $\ell(m)\leq k$. With this notation, we have
\begin{align*}
W_{1}^{(k)}\Pi_k f &=g(x_1,\dots,x_n) \sum_{i=1}^s \sum_{m\in [\beta_i]_k}  (\tc^k-\tc^{\ell(m)})\qc^{a_i} x_1^{a_i} m\\
&= g(x_1,\dots,x_n)\left(\sum_{i=1}^s \tc^k \qc^{a_i} x_1^{a_i} m_{\beta_i}[x_{n+1},\dots,x_k] -\sum_{i=1}^s \qc^{a_i} x_1^{a_i}  \sum_{m\in [\beta_i]_k}   \tc^{\ell(m)} m  \right)
\end{align*}
According to Definition \ref{Def of limit},
$$
\lim_k  W_{1}^{(k)}\Pi_k f = - \sum_{i=1}^s \qc^{a_i} x_1^{a_i} g(x_1,\dots,x_n)  \sum_{m\in [\beta_i]}   \tc^{\ell(m)} m \in \P_\infty^+,
$$
which proves our claim.
\end{proof}
\begin{ex}
The sequence $W_1^{(k)}(x_2^2+\cdots+x_k^2)$ from Example \ref{exp: p_2ii} has limit $-\qc^2\tc x_1^2$.
\end{ex}

%%%%%%%%%%%%%%%%%%%%%%%%%%%%%%%%%%%%%%%%%%%%%%%%%%%%%%%%%
\subsection{}\label{sec: ycal} We can now prove the following.
\begin{prop}\label{prop: Y}
Let $f\in \Pas^+$ and $i\geq 1$. The sequence $\widetilde{Y}_i^{(k)} \Pi_k f$ converges. We define
$$\gls{Ycal}: \Pas^+\to \P_\infty^+,\quad \Y_i f=\lim_{k} \widetilde{Y}_i^{(k)}\Pi_k f\in \P_\infty^{+}.$$
\end{prop}
\begin{proof}
Straightforward from Proposition \ref{prop: Z} and Lemma \ref{lemma: W}. More precisely, $\Y_i=\widetilde{Z}^{(\infty)}_i+W^{(\infty)}_i$.
\end{proof}

\begin{ex}\label{exp: p_2iii}
In continuation of Examples \ref{exp: p_2} and \ref{exp: p_2ii}, we have
$$
\Y_1 \cdot  p_2[\Xb_1]=-\qc^2 \tc x_1^2+\qc\tc(\tc-1)x_1 e_1[\Xb_1].
$$
As it can be seen, $p_2[\Xb_1]$ and $\Y_1 \cdot  p_2[\Xb_1]$ are elements of $\P(1)^+$.

Also, from Examples \ref{exp: x_2^2} and  \ref{exp: x_2^2ii} we have
$$
\Y_1 \cdot  x_2^2=\qc^2 \tc(\tc-1)x_1^2+\qc\tc(\tc-1)x_1 x_2-\qc\tc(\tc-1)^2 x_1 e_1[\Xb_2].
$$
In this case, $x_2^2$ and $\Y_1 x_2^2$ are elements of $\P(2)^+$.
\end{ex}

It turns out that the image of $\Y_i$ is contained in $\Pas^+$. We will use the following result.
\begin{lem}\label{Ytilde extension}
Let $f\in\mathbb{Q}(\tc,\qc)[x_{k+1},x_{k+2},\dots ]$ be a polynomial satisfying 
$$T_i f=f,\quad \textrm{ for }k+1\leq i\leq k+m.$$
Let $g=T_{k+m}^{-1}\dots T_{k+1}^{-1}T_k^{-1}(x_{k}^{s}f)$, for some $s\geq 0$. Then, 
$$T_i g=g \quad \textrm{ for }k\leq i\leq k+m-1.$$
\end{lem}
\begin{proof}
The claim is a consequence of the braid relations for the elements $T_i$.
\end{proof}

\begin{lem}\label{lem: P(k)stable}
$\P(k)^{+}$ is stable under the action of $\Y_1$.
Therefore, we have
$$\Y_i:\Pas^{+}\rightarrow \Pas^{+}.$$
\end{lem}

\begin{proof}
Let $f\in \P(k)^{+}$ (recall the definition of $\P(k)^{+}$ in \S\ref{sec: p(k)}). For any $m>k+1$ we have
$$\widetilde{Y}_1^{(m)} \Pi_m f=\tc^{m} \varpi_m T_{m-1}^{-1}\dots T_{1}^{-1}\Pi_m f.$$
Since $f\in \P(k)^{+}$, $\Pi_m f$ is fixed under the action of $T_{k+1},\dots ,T_{m-1}$. Now, we have
$$\varpi_m T_{m-1}^{-1}\dots T_{1}^{-1}\Pi_m f=\varpi_m T_{m-1}^{-1}\dots T_{k+1}^{-1}T_{k}^{-1}(T_{k-1}^{-1}\dots T_{1}^{-1}\Pi_m f),$$
and we may write 
$$T_{k-1}^{-1}\dots T_{1}^{-1}\Pi_m f=\sum_{j} c_j[x_1,\dots ,x_{k-1}]x_{k}^{s_j}f_j$$
as a finite sum, where each $f_j$ is a polynomial in $\mathbb{Q}(\tc,\qc)[x_{k+1},x_{k+2},\dots ]$ satisfying 
$$T_i f_j=f_j,\quad \textrm{ for }i=k+1,k+2,\dots ,m-1.$$
Applying {Lemma \ref{Ytilde extension}} for $T_{k-1}^{-1}\dots T_{1}^{-1}\Pi_m f$, we obtain that 
$ T_{m-1}^{-1}\dots T_{1}^{-1}\Pi_m f$ is fixed under the action of  the elements $T_{k},\dots ,T_{m-2}$. The relation \eqref{omega3 relation} implies that 
$\varpi_m T_{m-1}^{-1}\dots T_{1}^{-1}\Pi_m f$ is fixed under the action of $T_{k+1},\dots ,T_{m-1}$. Therefore, $\widetilde{Y}_1^{(m)} \Pi_m f$ is symmetric in $x_{k+1},\dots, x_m$ for all $m>k+1$. In conclusion, the limit
$$\displaystyle \lim_{m} \widetilde{Y}_i^{(m)}\Pi_m f\in \P(k)^+,$$
proving our claim.
\end{proof}

We can now state the following.

\begin{prop}\label{prop: limit action}
The space $\Pas^+$ carries an action of the limit operators $\T_i, \X_i, \Y_i$, $i\geq 1$.
\end{prop}
As we will show in Theorem \ref{thm: +standardrep}, these operators define an action of $\H^+$ on $\Pas^+$. 
\begin{rmk}
Following up on Remark \ref{rem: modCherednik},  it is important to note that on $x_i\Pas^+$ the action of $\Y_i$ is the stable limit of the action of the sequence of Cherednik operators $Y^{(k)}_i\in \H_k$.
\end{rmk}
%%%%%%%%%%%%%%%%%%%%%%%%%%%%%%%%%%%%%%%%%%%%%%%%%%%%%%%%%
\subsection{} In fact, we can obtain a more precise description of the action of $\Y_1$ on $\P(k)^+$. First, let us record the following technical result. We use the notation in \S \ref{sec: esf}.
\begin{lem}\label{finite DAHA lemma}
Let $x_1^{n}\in \P^+_{m+1}$. Then, 
$$\tc^{m}T_{m}^{-1}...T_{1}^{-1}x_1^{n}=\sum_{i=0}^{n-1}x_{m+1}^{n-i}h_i[(1-\tc)\overline{\mathbf{X}}_{m}].$$
\end{lem}
\begin{proof}
We will prove the result by induction on $m$. By direct computation we obtain
\begin{align*}
T_{1}^{-1}x_1^{n}&=\tc^{-1}x_2^{n}+(\tc^{-1}-1)(x_2^{n-1}x_1+x_2^{n-2}x_1^2+...+x_2x_1^{n-1})\\
	&=\tc^{-1}\sum_{i=0}^{n-1}x_{2}^{n-i}h_i[(1-\tc)x_1].
\end{align*}
Assume that our claim holds for $m-1$. Then, we have
\begin{align*}
\tc^{m}T_{m}^{-1}...T_{1}^{-1}x_1^{n}&=\tc T_{m}^{-1}\sum_{i=0}^{n-1}x_{m}^{n-i}h_i[(1-\tc)\overline{\mathbf{X}}_{m-1}]\\
	&=\sum_{i=0}^{n-1}\sum_{j=0}^{n-i-1}x_{m+1}^{n-i-j}h_j[(1-\tc)x_m] h_i[(1-\tc)\overline{\mathbf{X}}_{m-1}]\\
	&=\sum_{l=0}^{n-1} x_{m+1}^{n-l}\sum_{\substack{i+j=l\\i, j\geq 0}}h_j[(1-\tc)x_m] h_i[(1-\tc)\overline{\mathbf{X}}_{m-1}]\\
	&=\sum_{l=0}^{n-1} x_{m+1}^{n-l}h_l[(1-\tc)\overline{\mathbf{X}}_{m}],
\end{align*}
as expected.
\end{proof}

Recall that for all $k$, the multiplication map  $ \P_k^{+}\otimes \Sym[\mathbf{X}_k]  \cong\P(k)^+$ is an algebra isomorphism.
%By writing any $F[\mathbf{X}_{k}]\in \Sym[\mathbf{X}_{k}]$ as $F[\mathbf{X}_{k-1}-x_k]$ we see that the elements of  $\P(k)^+$ can be written as  finite sums of the form
%\begin{equation}\label{eq10bis}
%\sum f_i(x_1,\dots,x_{k-1})x_k^i G_i[\mathbf{X}_{k-1}].
%\end{equation}
\begin{prop}\label{Y1 action}
Let $n\geq 0$, $f(x_1,\dots,x_{k-1})\in \P_{k-1}^{+}$, and $G[\mathbf{X}_{k-1}]\in  \Sym[\mathbf{X}_{k-1}]$. We regard $$F=f(x_1,\dots,x_{k-1})x_k^n G[\mathbf{X}_{k-1}]$$ as an element of $\P(k)^+$. Then,
$$
\Y_1\T_1\cdots \T_{k-1} F =  \frac{\tc^k}{1-\tc} f(x_2,\dots,x_{k})G[\mathbf{X}_{k}+\qc x_1](h_n[(1-\tc )(\mathbf{X}_k+\qc x_1)]-h_n[(1-\tc )\mathbf{X}_k]).
$$
\end{prop}
\begin{proof}
For any $m>k+1$, we have
\begin{align*}\widetilde{Y}_1^{(m)} T_1\cdots T_{k-1} \Pi_m F&=\tc^{m} \varpi_m T_{m-1}^{-1}\dots T_{k}^{-1} f(x_1,\dots,x_{k-1})x_k^n G[\overline{\mathbf{X}}_{[k,m]}]\\
&=\tc^{k}  \varpi_m f(x_1,\dots,x_{k-1})G[\overline{\mathbf{X}}_{[k,m]}] \tc^{m-k} T_{m-1}^{-1}\dots T_{k}^{-1} x_k^n\\
&=\tc^{k}  \varpi_m f(x_1,\dots,x_{k-1})G[\overline{\mathbf{X}}_{[k,m]}] \sum_{i=0}^{n-1}x_{m}^{n-i}h_i[(1-\tc)\overline{\mathbf{X}}_{[k,m-1]}]\\
&=\tc^{k} f(x_2,\dots,x_{k})G[\overline{\mathbf{X}}_{[k+1,m]}+\qc x_1] \sum_{i=0}^{n-1}(\qc x_{1})^{n-i}h_i[(1-\tc)\overline{\mathbf{X}}_{[k+1,m]}].
\end{align*}
Therefore, $\displaystyle \lim_{m}\widetilde{Y}_i^{(m)}T_1\cdots T_{k-1} \Pi_m F$, equals
\begin{align*}
\tc^{k} f(x_2,\dots,x_{k})& G[\mathbf{X}_{k}+\qc x_1]  \sum_{i=0}^{n-1}(\qc x_{1})^{n-i}h_i[(1-\tc){\mathbf{X}}_{k}]\\
&=\frac{\tc^{k}}{1-\tc} f(x_2,\dots,x_{k})G[{\mathbf{X}}_{k}+\qc x_1](h_n[(1-\tc )(\mathbf{X}_k+\qc x_1)]-h_n[(1-\tc )\mathbf{X}_k]), 
\end{align*}
proving our claim.
\end{proof}

%%%%%%%%%%%%%%%%%%%%%%%%%%%%%%%%%%%%%%%%%%%%%%%%%%%%%%%%%
\subsection{} We establish the following result in preparation for the proof of Theorem \ref{thm: +standardrep}.
\begin{lem}\label{finite lemma}
Let $f\in \Pas^+$. Then,
$$\lim_{k} [\widetilde{Y}_{i}^{(k)},\widetilde{Y}_{j}^{(k)}]\Pi_k f=0.$$
\end{lem}
\begin{proof}
First note that we can apply the following two relations recursively
$$[\widetilde{Y}_{i}^{(k)},\widetilde{Y}_{j}^{(k)}]=\tc^{-1}T_{j-1}[\widetilde{Y}_{i}^{(k)},\widetilde{Y}_{j-1}^{(k)}]T_{j-1}, \textrm{ for } i>j,$$
$$[\widetilde{Y}_{1}^{(k)},\widetilde{Y}_{i}^{(k)}]=\tc^{-1}T_{i-1}[\widetilde{Y}_{1}^{(k)},\widetilde{Y}_{i-1}^{(k)}]T_{i-1}, \textrm{ for } i>2.$$
Hence it suffices to prove the result for $[\widetilde{Y}_{1}^{(k)},\widetilde{Y}_{2}^{(k)}]$.

By Remark \ref{rem: [y,y]} we have
$$[\widetilde{Y}_{1}^{(k)},\widetilde{Y}_{2}^{(k)}]=\tc^{2k-1}\gamma_k T_{k-1}^{-1}...T_1^{-1}T_{k-1}^{-1}...T_2^{-1}=\tc^{2k-1}\gamma_k T_{k-2}^{-1}...T_1^{-1}T_{k-1}^{-1}...T_1^{-1}.$$
Let $f(x_1,...,x_m)F[\Xb]\in\P(k)^{+}$, where $F[\Xb]$ is fully symmetric and non-zero. Then for $k>m$ we have
\begin{align*}
\gamma_k T_{k-1}^{-1}...T_1^{-1}T_{k-1}^{-1}...T_2^{-1}\Pi_k f(x_1,...,x_m)F[\Xb] &= \gamma_k T_{k-1}^{-1}...T_1^{-1}T_{k-1}^{-1}...T_2^{-1} f(x_1,...,x_m)F[\overline{\mathbf{X}}_{k}]\\
	&=\gamma_k F[\overline{\mathbf{X}}_{k}](T_{k-1}^{-1}...T_1^{-1}T_{k-1}^{-1}...T_2^{-1} f(x_1,...,x_m))\\
	&=0.
\end{align*}
For the last equality we used \eqref{eq: gamma}. Therefore, in this case,
$$\lim_{k\rightarrow\infty} [\widetilde{Y}_{i}^{(k)},\widetilde{Y}_{j}^{(k)}]\Pi_k f(x_1,...,x_m)F[\Xb]=0.$$

It remains to compute the limit for $f(x_1,...,x_m)\in \P_m^+$. Let $k>m+1$. Without loss of generality, we may assume that
$$T_{m-2}^{-1}...T_1^{-1}T_{m-1}^{-1}...T_1^{-1} f(x_1,...,x_m)$$
is a monomial of the form $x_m^{n}x_{m-1}^{n^\prime}g(x_1,...,x_{m-2})$. If $n>0$,  we have
\begin{align*}
[\widetilde{Y}_{1}^{(k)},\widetilde{Y}_{2}^{(k)}] f(x_1,...,x_m)
&= \gamma_k T_{k-2}^{-1}...T_1^{-1}T_{k-1}^{-1}...T_1^{-1} f(x_1,...,x_m)\\
&=\gamma_k T_{k-2}^{-1}...T_{m-1}^{-1}T_{k-1}^{-1}...T_{m}^{-1}(T_{m-2}^{-1}...T_1^{-1}T_{m-1}^{-1}...T_1^{-1} f(x_1,...,x_m))\\
&=\gamma_k T_{k-2}^{-1}...T_{m-1}^{-1}T_{k-1}^{-1}...T_{m}^{-1}x_m^{n}x_{m-1}^{n^\prime}g(x_1,...,x_{m-2})\\
&= g(x_1,...,x_{m-2}) \gamma_k  T_{k-2}^{-1}...T_{m-1}^{-1} x_{m-1}^{n^\prime} T_{k-1}^{-1}...T_{m}^{-1}x_m^{n}.
\end{align*}

If $n=n^\prime=0$ then $[\widetilde{Y}_{1}^{(k)},\widetilde{Y}_{2}^{(k)}]f(x_1,...,x_m)=0$ by \eqref{eq: gamma}. If $n=0$ and $n^\prime>0$, then by Lemma \ref{finite DAHA lemma}
\begin{align*}
\gamma_k  T_{k-2}^{-1}...T_{m-1}^{-1}x_{m-1}^{n^\prime}
&= \tc^{m-k}\gamma_k  \sum_{i=0}^{n^\prime -1}  x_{k-1}^{n^\prime -i} h_i[(1-\tc)\overline{\mathbf{X}}_{[m-1,k-2]}]\\
&= \tc^{m-k} \sum_{i=0}^{n^\prime -1}  h_i[(1-\tc)\overline{\mathbf{X}}_{{[m+1,k]}}]\gamma_k    x_{k-1}^{n^\prime -i}  \\
&=\tc^{m-k}(1-\tc) \sum_{i=0}^{n^\prime -1}  {\qc^{n^\prime-i}}h_i[(1-\tc)\overline{\mathbf{X}}_{{[m+1,k]}}](x_1^{n^\prime-i-1}x_2+\cdots+x_1x_2^{n^\prime-i-1}).
\end{align*}
Therefore, $\displaystyle\lim_{k} [\widetilde{Y}_{1}^{(k)},\widetilde{Y}_{2}^{(k)}]f(x_1,...,x_m)=0$.

If $n^\prime=0$ and $n>0$, then by Lemma \ref{finite DAHA lemma}
\begin{align*}
\gamma_k  T_{k-2}^{-1}...T_{m-1}^{-1}T_{k-1}^{-1}...T_{m}^{-1}x_{m}^{n}
&= \tc^{m-k}\gamma_k  T_{k-2}^{-1}...T_{m-1}^{-1} \sum_{i=0}^{n -1}  x_k^{n -i} h_i[(1-\tc)\overline{\mathbf{X}}_{[m,k-1]}].
\end{align*}
By writing 
$$
h_i[(1-\tc)\overline{\mathbf{X}}_{[m,k-1]}]=\sum_{j=0}^i h_{i-j}[(1-\tc)\overline{\mathbf{X}}_{[m-1,k-1]}]h_j[(\tc-1)x_{m-1}]
$$
and using again Lemma \ref{finite DAHA lemma} for $T_{k-2}^{-1}...T_{m-1}^{-1} x_{m-1}^j$ as well as \eqref{eq: gamma}, we write  $\gamma_k  T_{k-2}^{-1}...T_{m-1}^{-1}T_{k-1}^{-1}...T_{m}^{-1}x_{m}^{n}$
as
\begin{align*}
%\gamma_k  T_{k-2}^{-1}...T_{m-1}^{-1}T_{k-1}^{-1}...T_{m}^{-1}x_{m}^{n}
\tc^{m-k} \sum_{i=0}^{n -1}  \gamma_k    x_k^{n -i}  &h_{i}[(1-\tc)\overline{\mathbf{X}}_{[m-1,k-1]}]\\
&= \tc^{m-k}(\tc -1) \sum_{i=0}^{n -1} {\qc^{n-i}}  h_i[(1-\tc)\overline{\mathbf{X}}_{{[m+1,k]}}](x_1^{n-i-1}x_2+\cdots+x_1x_2^{n-i-1}).
\end{align*}
Therefore, $\displaystyle \lim_{k} [\widetilde{Y}_{1}^{(k)},\widetilde{Y}_{2}^{(k)}]f(x_1,...,x_m)=0$.

For the last case, $n,n^\prime>0$, proceeding as in the previous case we obtain that 
$$\gamma_k  T_{k-2}^{-1}...T_{m-1}^{-1} x_{m-1}^{n^\prime}  T_{k-1}^{-1}...T_{m}^{-1}x_{m}^{n}$$
equals
$$
\tc^{m-k}\gamma_k  T_{k-2}^{-1}...T_{m-1}^{-1}x_{m-1}^{n^\prime} \sum_{i=0}^{n -1}  x_k^{n -i} \sum_{j=0}^i h_{i-j}[(1-\tc)\overline{\mathbf{X}}_{[m-1,k-1]}]h_j[(\tc-1)x_{m-1}].
$$
Lemma \ref{finite DAHA lemma} for $T_{k-2}^{-1}...T_{m-1}^{-1} x_{m-1}^j$ and \eqref{eq: gamma} imply that $[\widetilde{Y}_{1}^{(k)},\widetilde{Y}_{2}^{(k)}]f(x_1,...,x_m)=0$.
\end{proof}

%%%%%%%%%%%%%%%%%%%%%%%%%%%%%%%%%%%%%%%%%%%%%%%%%%%%%%%%%
\subsection{}
We are now ready to prove our main result.

 \begin{thm} \label{thm: +standardrep}
The operators $\T_i$, $\X_i$, and $\Y_i$, $i\geq 1$, define a $\H^+$-module structure on $\Pas^+$.
\end{thm}
\begin{proof}
The relations that hold in the algebra $\Ht_k$ also hold for the corresponding limit operators by the repeated application of Corollary \ref{cor: AB}. Recall that by Remark \ref{rem: crossrel} the relation \eqref{XY cross relations} and the 
first two relations in \eqref{Y relations} also hold inside the algebras $\Ht_k$. The only relations that do not transfer directly from those in $\Ht_k$ are the commutation relations between $\Y_i$ and $\Y_j$. However, we do have 
\begin{align*}
[\widetilde{Y}^{(k)}_{i},\widetilde{Y}^{(k)}_{j}]=\tc^{-1}T_{j-1}[\widetilde{Y}^{(k)}_{i},\widetilde{Y}^{(k)}_{j-1}]T_{j-1}, \quad i>j,\\
[\widetilde{Y}^{(k)}_{1},\widetilde{Y}^{(k)}_{i}]=\tc^{-1}T_{i-1}[\widetilde{Y}^{(k)}_{1},\widetilde{Y}^{(k)}_{i-1}]T_{i-1}, \quad i>2.
\end{align*}
These, by application of Corollary \ref{cor: AB} imply the same relations for the limit operators
\begin{align*}
[\Y_{i},\Y_{j}]&=\tc^{-1}\T_{j-1}[\Y_{i},\Y_{j-1}]\T_{j-1}, \quad i>j,\\
[\Y_{1},\Y_{i}]&=\tc^{-1}\T_{i-1}[\Y_{1},\Y_{i-1}]\T_{i-1}, \quad i>2.
\end{align*}
Therefore, for any $i,j$, the commutativity of $\Y_i$ and $\Y_j$ would follow from the commutativity of $\Y_1$ and $\Y_2$. Fix $f\in \Pas^+$. Then, by Proposition \ref{prop: continuity} and Corollary \ref{cor: AB},
\begin{align*}
[\Y_1,\Y_2]f=\lim_k [\widetilde{Y}^{(k)}_{1},\widetilde{Y}^{(k)}_{2}]\Pi_k f,
\end{align*}
which is $0$ by Lemma \ref{finite lemma}.
\end{proof}
We call this representation the standard representation of $\H^+$.  As with the standard representation of $\H^-$, we expect this representation to be faithful.

%%%%%%%%%%%%%%%%%%%%%%%%%%%%%%%%%%%%%%%%%%%%%%%%%%%%%%%%%
\section{The double Dyck path algebra}\label{sec: ddpa}

\subsection{}\label{subsec: ddpa} The main result of this section is a explicit connection between the standard representation of $\H^+$ and the double Dyck path algebra and its standard representation. The double Dyck path algebra, denoted in the 
literature by $\mathbb{A}_{\qc,\tc}$ is an algebraic structure discovered by Carlsson and Mellit. It plays a critical role in the proof \cite{CM} of the Compositional Shuffle Conjecture \cites{HHLRU, HMZ}  and the proof \cite{Me} of the more 
general Compositional $(km,kn)$-Shuffle Conjecture \cites{GN, BGLX}.

To facilitate the comparison, we will switch the role of the parameters $\qc$ and $\tc$ in the original definition of the double Dyck path algebra. Therefore, the definitions below correspond to $\Atq$. We will first define two quivers 
\gls{Q} and introduce some conventions.

The quiver $\dot{\textbf{Q}}$ is defined to be the quiver with vertex set $\mathbb{Z}_{\geq 0}$, and, for all $k\in \mathbb{Z}_{\geq 0}$, arrows $d_{+}$ from $k$ to $k+1$, arrows $d_{-}$ from $k+1$ to $k$, and for $k\geq2$ loops 
$T_1,...,T_{k-1}$ from $k$ to $k$. Note that, to keep the notation as simple as possible, the same label ($T_i$, \gls{d}) is used to denote many arrows. To eliminate the possible confusion we adopt the following convention.

\begin{conv}
In all expressions involving paths in $\dot{\textbf{Q}}$, unless specified otherwise, we assume that all the expressions involve non-zero paths (that is, the constituent arrows concatenate correctly to produce a non-zero path) that start at 
the node $k$ (fixed, but arbitrary).
\end{conv}

Then quiver $\ddot{\textbf{Q}}$ is the quiver with vertex set $\mathbb{Z}_{\geq 0}$, and, for all $k\in \mathbb{Z}_{\geq 0}$, arrows $d_{+}$ and $d_{+}^{*}$ from $k$ to $k+1$, arrows $d_{-}$ from $k+1$ to $k$, and loops 
$T_1,...,T_{k-1}$ from $k$ to $k$. We will adopt the same labelling convention for paths in $\ddot{\textbf{Q}}$.

\begin{dfn}\label{def: dpa}
The Dyck path algebra \gls{At} is defined as the quiver path algebra of $\dot{\textbf{Q}}$ modulo the following relations:
\begin{subequations}\label{DPA}
    \begin{equation}\label{T relation}
      \begin{gathered}
      T_{i}T_{j}=T_{j}T_{i}, \quad |i-j|>1,\\
      T_{i}T_{i+1}T_{i}=T_{i+1}T_{i}T_{i+1}, \quad 1\leq i\leq k-2,
      \end{gathered}%\tag{2.1.a}
    \end{equation}
    \begin{equation}\label{T quadratic}
     (T_{i}-1)(T_{i}+\tc)=0, \quad 1\leq i\leq k-1,
    \end{equation}
    \begin{equation}\label{missing d- relation}
    d_-^2T_{k-1}=d_-^2,\quad T_id_-=d_-T_i, \quad 1\leq i\leq k-2.
    \end{equation}
    \begin{equation}\label{d+ relation}
        %\begin{gathered}
          T_1 d_{+}^2 = d_{+}^2,\quad d_{+}T_i = T_{i+1}d_{+}, \quad 1\leq i\leq k-1,%\tag{2.1.b}
        %\end{gathered}
    \end{equation}
    \begin{equation}\label{d- relation}
      \begin{gathered}
         d_{-}[d_{+},d_{-}]T_{k-1} = \tc[d_{+},d_{-}]d_{-}, \quad (k\geq 2)\\
         T_{1}[d_{+},d_{-}]d_{+} = \tc d_{+}[d_{+},d_{-}],\quad (k\geq 1).
      \end{gathered}%\tag{2.1.c}
    \end{equation}
\end{subequations}
\end{dfn}

\begin{dfn}\label{def: ddpa}
  The double Dyck path algebra \gls{Atq} is defined as the quiver path algebra of $\ddot{\textbf{Q}}$ modulo the relations \eqref{T relation}, \eqref{T quadratic}, \eqref{d+ relation} and \eqref{d- relation}
  and the following additional relations:
    \begin{subequations}\label{DDPA}
      \begin{equation}\label{d+* relation}
        \begin{gathered}
            T_1 (d_{+}^{*})^2 = (d_{+}^{*})^2,\quad d_{+}^{*}T_i = T_{i+1}d_{+}^{*},\quad 1\leq i\leq k-1,
        \end{gathered}%\tag{2.1.d}
      \end{equation}
      \begin{equation}\label{d-2 relation}
        \begin{gathered}
            \tc d_{-}[d_{+}^{*},d_{-}] = [d_{+}^{*},d_{-}]d_{-}T_{k-1},\quad (k\geq 2)\\
            \tc[d_{+}^{*},d_{-}]d_{+}^{*} =  T_{1}d_{+}^{*}[d_{+}^{*},d_{-}]\quad (k\geq 1),
        \end{gathered}%\tag{2.1.e}
      \end{equation}
    %\item the cross relations:
    \begin{equation}\label{cross relation}
     % \begin{gathered}
         d_{+}z_{i}=z_{i+1}d_{+},\quad d_{+}^{*}y_{i}=y_{i+1}d_{+}^{*},\quad 1\leq i\leq k-1,
         \end{equation}
         \begin{equation}\label{last}
        z_1 d_{+}=-\qc\tc^{k+1}y_1d_{+}^{*}.
    %  \end{gathered}%\tag{2.1.f}
    \end{equation}
        \end{subequations}
The notation \gls{y} and \gls{z}  refers to the following operators corresponding to loops at $k$, $1\leq i \leq k$:
    \begin{equation*}
      \begin{gathered}
        y_1=\frac{1}{\tc^{k-1}(\tc-1)}[d_{+},d_{-}]T_{k-1}\cdots T_1,\\
        y_{i+1}=\tc T_{i}^{-1}y_i T_{i}^{-1},\quad 1\leq i \leq k-1,\\
        z_1=\frac{\tc^k}{1-\tc}[d_{+}^{*},d_{-}]T_{k-1}^{-1}\cdots T_1^{-1},\\
        z_{i+1}=\tc^{-1}T_i z_i T_i,\quad 1\leq i \leq k-1.
      \end{gathered}
    \end{equation*}
\end{dfn}

\begin{rmk}
There exists an $\Rat$-algebra involution of $\Atq$ defined by
$$T_i^{-1}\mapsto T_{i},\ d_{-}\mapsto d_{-},\ d_{+}^{*}\mapsto d_{+},\ \ d_{+}\mapsto d_{+}^{*},$$
and inverts the parameters $\tc$ and $\qc$.
We regard $\mathbb{A}_{\tc}$ as a subalgebra of $\Atq$, and denote its image under the involution by $\mathbb{A}_{\tc^{-1}}$.
\end{rmk}

%%%%%%%%%%%%%%%%%%%%%%%%%%%%%%%%%%%%%%%%%%%%%%%%%%%%%%%%%
\subsection{} 
For the following result we refer to  \cite{CM}*{Lemma 5.5} (see also \cite{Me}*{Proposition 3.1}).

\begin{prop}\label{prop: ddpa}
 The loops $y_i$'s and $z_i$'s satisfy the following relations:
 \begin{subequations}
  \begin{align}\label{y1 relation}
      \begin{split}
        y_i T_j &= T_j y_i,\quad i\neq j,j+1 \\
        y_{i+1} &=\tc T_{i}^{-1}y_i T_{i}^{-1},\quad 1\leq i \leq k-1,\\
        y_i y_j &= y_j y_i, \quad  1\leq i,j\leq k,
      \end{split}%\tag{2.2.a}
  \end{align}
  \begin{align}\label{y2 relation}
      \begin{split}
        y_i d_{-} &= d_{-} y_i,\quad  1\leq i\leq k-1 \\
        d_{+} y_i &= T_1...T_i y_i T_{i}^{-1}...T_1^{-1}d_{+}, \quad 1\leq i\leq k,
      \end{split}%\tag{2.2.b}
  \end{align}
  \begin{align}\label{z1 relation}
      \begin{split}
        z_i T_j &= T_j z_i,\quad i\neq j,j+1 \\
        z_{i+1} &= \tc^{-1}T_i z_i T_i,\quad 1\leq i \leq k-1,\\
        z_i z_j &= z_j z_i, \quad  1\leq i,j\leq k,
      \end{split}%\tag{2.2.c}
  \end{align}
  \begin{align}\label{z2 relation}
      \begin{split}
        z_i d_{-} &= d_{-} z_i,\quad  1\leq i\leq k-1 \\
        d_{+}^{*} z_i &=  T_{1}^{-1}...T_i^{-1} z_i T_i...T_1 d_{+}^{*}, \quad 1\leq i\leq k,
      \end{split}%\tag{2.2.d}
  \end{align}
 \end{subequations}
\end{prop}

\begin{rmk}
Note that the relations (\ref{y1 relation}) match the generating relations for the algebra  $\AHA_k^{+}$. Similarly,  $T_1^{-1},...,T_{k-1}^{-1}$ and $z_1,...,z_k$ generate a copy of $\AHA_k^{+}$ with parameter $\tc^{-1}$.
\end{rmk}

%%%%%%%%%%%%%%%%%%%%%%%%%%%%%%%%%%%%%%%%%%%%%%%%%%%%%%%%%
\subsection{} \label{sec: ddparep}
The algebra $\Atq$ comes with a canonical quiver representation. The representation emerged first, from the analysis of certain operators acting on generating functions of Dyck paths. The algebra itself is a abstract 
formalization of the properties of the relevant operators. We will use the representation defined in \cite{Me}*{Proposition 3.2, Proposition 3.3}. The relationship between this action and the original action defined in \cite{CM} is explained in 
\cite{Me}*{\S 3.3}. 

Let $V_k=\mathbb{Q}(\tc,\qc)[y_1,...,y_k]\otimes\textrm{Sym}[\Xb]$, and  denote $V_{\bullet}=(V_k)_{k\geq 0}$. It is important to note that $V_{\bullet}$ is naturally equipped with a $\lambda$-ring structure.
The symmetric group $S_k$ acts on $V_k$ by permuting the variables $y_i$.  We denote by $\zeta_k$ the algebra morphism that acts trivially 
$\Sym[\Xb]$ 
and acts on $\mathbb{Q}(\tc,\qc)[y_1,...,y_k]$ as
$$
\gls{zeta} f(y_1,\dots,y_{k-1},y_k)=f(y_2,\dots,y_k,\qc y_1).
$$
Finally, we denote by $\gls{ct} F$ the constant term of $F$ with respect to $y_k$.

\begin{thm}\label{thm: CMrep}
 The following operators define a quiver representation of $\Atq$ on $V_{\bullet}$:
  \begin{align}\label{Aqt representation}
      \begin{split}
        T_i F &= s_i F+(1-\tc)y_i\frac{F-s_i F}{y_i-y_{i+1}}, \\
        d_{-}F &= \ct_{y_k}(F[\Xb-(\tc-1)y_k]\textnormal{Exp}[-y_k^{-1}\Xb]), \\
        d_{+}F &= -T_1...T_{k}(y_{k+1}F[\Xb+(\tc-1)y_{k+1}],\\
        d_{+}^{*}F &= \zeta_k F[\Xb+(\tc-1)y_{k+1}],
      \end{split}%\tag{2.3}
  \end{align}
The formulas above represent the actions of the arrows originating at node $k$.
\end{thm}

We call this representation  the standard representation of the double Dyck path algebra $\Atq$.

An important result is the computation of the action of the loops $y_i$ in the standard representation. We refer to \cite{CM}*{Lemma 5.4, Lemma 5.5} (see also \cite{Me}*{Proposition 3.2})

\begin{prop}
The loops $y_i$, $1\leq i\leq k$, act on $V_k$ as multiplication by $y_i$.
\end{prop}

The action of the $z_i$ operators is more complicated. We record below some examples that show in particular that the action of $z_1$ is not compatible with the canonical inclusion $V_k\subset V_{k+1}$.
\begin{ex}\label{exp: x_2^2iv}
Let $y_2^2\in V_2$. We have,
$$z_1 \cdot  y_2^2 =\qc^2 \tc(\tc-1)y_1^2+\qc\tc(\tc-1)y_1 y_2-\qc\tc(\tc-1)y_1 e_1[\Xb].$$
Let $y_2^2\in V_3$. We have,
$$z_1 \cdot  y_2^2 =\qc^2 \tc(\tc-1)y_1^2+\qc\tc(\tc-1)y_1 y_2-\qc\tc(\tc-1)^2 y_1y_3-\qc\tc(\tc-1)y_1 e_1[\Xb].$$
More generally, for $y_2^2\in V_k$, $k\geq 3$, we have,
$$z_1 \cdot  y_2^2 =\qc^2 \tc(\tc-1)y_1^2+\qc\tc(\tc-1)y_1 y_2-\qc\tc(\tc-1)^2 (y_1y_3+\cdots+y_1y_k)-\qc\tc(\tc-1)y_1 e_1[\Xb].$$
\end{ex}
\begin{ex}\label{exp: p_2iv}
Let $p_2[\Xb]\in V_1$. We have,
$$
z_1 \cdot  p_2[\Xb]=\qc^2 \tc(1-\tc^2)  y_1^2-\qc\tc(1-\tc^2)y_1 e_1[\Xb].
$$
Let $p_2[\Xb]\in V_2$. We have,
$$
z_1 \cdot  p_2[\Xb]= \qc^2 \tc^2(1-\tc^2)y_1^2  - \qc\tc^2(1-\tc^2)y_1 e_1[\Xb].
$$
\end{ex}

%%%%%%%%%%%%%%%%%%%%%%%%%%%%%%%%%%%%%%%%%%%%%%%%%%%%%%%%%
\subsection{}\label{sec: ddpaP}
We will use the standard representation of $\H^+$ to construct a quiver representation of $\Atq$. Let
$$\gls{Pbullet}=(\P(k)^+)_{k\geq 0}.$$
For $k\geq 0$, recall that we denote by $\H^+(k)$ the subalgebra of $\H^+$ generated by $T_i$, $1\leq i\leq k-1$, and $X_i$, and $Y_i$, $1\leq i\leq k$. From Lemma \ref{lem: P(k)stable} we know that each  $\P(k)^+$ is stable under the action of $\H^+(k)$ 
through the standard representation of $\H^+$. Recall that, for all $k$, the multiplication map  $ \P_k^{+}\otimes \Sym[\mathbf{X}_k]  \cong\P(k)^+$ is an algebra isomorphism.

The elements $$\widetilde{\omega}_k^{-1}, \omega^{-1}_{k}\in \H_k^+$$
act on $\P_k^+$ via the standard representation of $\H_k^+$ (see Proposition \ref{laurent rep}). We extend their action to $\P(k)^+$ as  $\Sym[\mathbf{X}_k]$-linear maps. Let $$\gls{iotak}: \P(k)^+\to\P(k+1)^+$$ be the canonical inclusion 
map. Denote
$$\gls{par} =-\widetilde{\omega}^{-1}_{k+1}\iota(k): \P(k)^+\to\P(k+1)^+ \quad \text{and}\quad \gls{park} =\omega^{-1}_{k+1}\iota(k): \P(k)^+\to\P(k+1)^+.$$

%%%%%%%%%%%%%%%%%%%%%%%%%%%%%%%%%%%%%%%%%%%%%%%%%%%%%%%%%
\subsection{}\label{sec: d-}

Recall that the Hall-Littlewood symmetric functions \gls{HL} are a distinguished basis of the ring of symmetric functions $\Sym[\Xb]$, indexed by partitions $\lambda$. There is a remarkable family of linear operators \gls{B},
$n\geq 0$, on $\Sym[\Xb]$, defined as follows. $\Bo_\infty$ is the operator of left multiplication by the elementary symmetric function $e_1[\Xb]=\Xb$ and $\Bo_0$ is the operator defined by
$$\Bo_0 P_{\mu}(\Xb,\tc^{-1})=\tc^{\ell(\mu)}P_{\mu}(\Xb,\tc^{-1}).$$
For $n\geq 0$, let $\Bo_{n+1}:=[\Bo_\infty,\Bo_n]$.

 The operator $$\gls{parmin}:\P(k)^+\to\P(k-1)^+$$
 is defined to be the $\P_{k-1}^+$-linear map which, on elements of the form $x_k^nF[\mathbf{X}_{k}]$ acts as
$$\partial^{-}_k (x_k^nF[\mathbf{X}_{k}])= \Bo_n F[\mathbf{X}_{k-1}].$$
%where $\tau_k$ denotes the alphabet shift $\mathbf{X}_{k}\mapsto \mathbf{X}_{k-1}$ (or $ x_{i+1}\mapsto x_i\ \textrm{ for all } i\geq k$).

%%%%%%%%%%%%%%%%%%%%%%%%%%%%%%%%%%%%%%%%%%%%%%%%%%%%%%%%%
\subsection{}
 The operators $\Bo_n$, $n\geq 0$, are creation operators for the Hall-Littlewood symmetric functions. They are (modulo a change of variable) the vertex operators in \cite{Jing} (see also \cite{Mac}*{\S III.5, Exp. 8}). They are particular 
cases of more general operators (as in \cites{GHT, BGHT}) depending of both parameters $\qc, \tc$, which can be described more explicitly using plethystic substitution. In our case, 
$$\Bo_n F[\Xb]=(F[\Xb-z^{-1}]\textrm{Exp}[-(\tc-1)z\Xb])_{\bigr|{z^r}},$$
where ${}_{\bigr|{z^r}}$ denotes the coefficient of $z^r$ in the indicated expression. The expression for $\Bo_0$ implies the specified formula for $\Bo_n$, $n\geq 1$ (see, e.g., \cite{GHT}*{Proposition 1.4}). $\Bo_0$ is the $\qc=0$ 
specialization of the operator $\Delta^\prime$ in \cite{Ha}*{(2.10)}. The fact that the Hall-Littlewood symmetric functions are eigenfunctions of this operator is proved in \cite{Ha}*{Corrolary 2.3}. This leads to the following compact 
expression for $\partial_k^-$. Let $f(x_1,\dots,x_k)\in \P_k^+$ and $F[\mathbf{X}_{k}]\in \Sym[\mathbf{X}_{k}]$. Then,
\begin{equation}\label{eq6}
\partial^{-}_k f(x_1,...,x_k)F[\mathbf{X}_{k}]=\tau_k \ct_{x_k}(f(x_1,...,x_k)F[\mathbf{X}_{k}-x_k]\textrm{Exp}[-(\tc-1)x_k^{-1}\mathbf{X}_{k}]),
\end{equation}
where $\tau_k$ denotes the alphabet shift $\mathbf{X}_{k}\mapsto \mathbf{X}_{k-1}$ (or $x_{i+1}\mapsto x_i\ \textrm{ for all } i\geq k$). This description makes it clear that, if $F[\mathbf{X}_{k-1}]\in \Sym[\mathbf{X}_{k-1}]$ then
\begin{equation}\label{eq5}
\partial^{-}_k F[\mathbf{X}_{k-1}]= F[\mathbf{X}_{k-1}].
\end{equation}
By writing any $F[\mathbf{X}_{k}]\in \Sym[\mathbf{X}_{k}]$ as $F[\mathbf{X}_{k-1}-x_k]$ we see that the elements of  $\P(k)^+$ can be written as  finite sums of the form
\begin{equation}\label{eq10}
\sum f_i(x_1,\dots,x_{k-1})x_k^i G_i[\mathbf{X}_{k-1}].
\end{equation}
By \eqref{eq6}, on such an expression, $\partial_k^-$ acts as
\begin{equation}\label{eq7}
\partial_k^- \sum_i f_i(x_1,\dots,x_{k-1})x_k^i G_i[\mathbf{X}_{k-1}] = \sum_i f_i(x_1,\dots,x_{k-1}) G_i[\mathbf{X}_{k-1}]  \partial_k^- x_k^i
\end{equation}
%%%%%%%%%%%%%%%%%%%%%%%%%%%%%%%%%%%%%%%%%%%%%%%%%%%%%%%%%
\subsection{} To facilitate the comparison between the operators $\X_i$, $\Y_i$ and $[\partial,\partial^-]$, $[\partial^*, \partial^-]$ we record the following formulas.

\begin{lem} For any $k\geq 1$ we have
\begin{equation}\label{eq8}
\partial_{k-1}\partial_k^- - \partial_{k+1}^-\partial_k =(\tc-1)\widetilde{\omega}^{-1}_{k}.
\end{equation}
\end{lem}

\begin{proof}
The proof of this equality is identical to the one in the proof of \cite{CM}*{Lemma 5.4}; we include a brief explanation for the reader's convenience. Using the relations \eqref{missing d- relation}, we have
\begin{align*}
\partial_{k-1}\partial_k^- - \partial_{k+1}^-\partial_k &= -  \T_1\cdots \T_{k-1} \X_{k} \partial_k^- + \T_1 \T_2\cdots  \T_{k-1} \partial_{k+1}^-\T_{k} \X_{k+1}\\
&=  \T_1\cdots \T_{k-1} \X_{k} (-\partial_k^- + \partial_{k+1}^- \T_k^{-1}).
\end{align*}
The operator $-\partial_k^- + \partial_{k+1}^- \tc \T_k^{-1}$  acts on $\P(k)^+$ as scaling by $(\tc-1)$. By \eqref{eq7} this only needs to be verified for the action on $x_k^n$, $n\geq 0$. 
We have,
\begin{align*}
\partial_{k+1}^- \tc \T_k^{-1}x_k^{n}- \partial_k^-x_k^n &= \partial_{k+1}^- x_{k+1}^n+(1-\tc)\sum_{i=1}^{n-1}x_k^i \partial_{k+1}^- x_{k+1}^{n-i} - \partial_k^-x_k^n\\
&= h_n[(1-\tc)\mathbf{X}_k]+(1-\tc)\sum_{i=1}^{n-1}x_k^i h_{n-i}[(1-\tc)\mathbf{X}_k] - h_n[(1-\tc)\mathbf{X}_{k-1}]\\
&=-(1-\tc)x_k^n.
\end{align*}
This proves our claim.
\end{proof}

\begin{lem}
Let $k\geq 1$, $n\geq 0$, $f(x_1,\dots,x_{k-1})\in \P_{k-1}^{+}$, and $G[\mathbf{X}_{k-1}]\in  \Sym[\mathbf{X}_{k-1}]$. We regard $$F=f(x_1,\dots,x_{k-1})x_k^n G[\mathbf{X}_{k-1}]$$ as an element of $\P(k)^+$. Then,
\begin{equation}\label{eq11}
(\partial^*_{k-1}\partial_k^- - \partial_{k+1}^-\partial^*_k) F =f(x_2,\dots,x_{k})G[\mathbf{X}_{k}+\qc x_1](h_n[(1-\tc )(\mathbf{X}_k+\qc x_1)]-h_n[(1-\tc )\mathbf{X}_k]).
\end{equation}
\end{lem}
\begin{proof}
Straightforward from \eqref{eq5} and the definition of $\partial^*$, $\partial^-$.
\end{proof}
%%%%%%%%%%%%%%%%%%%%%%%%%%%%%%%%%%%%%%%%%%%%%%%%%%%%%%%%%
\subsection{} We are now ready to describe the $\Atq$-module structure on $\P_\bullet$.

\begin{thm}\label{thm: quiverrep}
The map that sends $T_i$, $d_+$, $d_+^{*}$, and $d^-$ to  $\T_i$, $\partial$, $\partial^*$, and $\partial^-$, respectively,
defines a  $\Atq$-module structure on $\P_\bullet$. Under this action, the operators $y_i$, $z_i$ act as $\X_i$, $\Y_i$.
\end{thm}
It is important to note that while the operators corresponding to arrows connecting different nodes are \emph{local} (i.e. dependent on $k$), the operators $\T_i$, $\X_i$, $\Y_i$ that correspond to loops are \emph{global} (i.e.
independent of $k$, as they are restrictions of operators on $\Pas^+$).

\begin{proof}
The fact that the operator 
$$
y_1=\frac{1}{\tc^{k-1}(\tc-1)}[d_{+},d_{-}]T_{k-1}\dots T_1,
$$
acts on $\P(k)^+$ as $\X_1$ follows from \eqref{eq8} and the first expression for $\widetilde{\omega}^{-1}_k$ in Remark \ref{rem: otilderels}.  Therefore, for all $i\geq 1$, the operator $y_i$ acts as $\X_i$.

Furthermore, from \eqref{eq10}, \eqref{eq11},  and Proposition \ref{Y1 action} we obtain that 
$$
z_1T_{1}\cdots T_{k-1}=\frac{\tc^k}{1-\tc}[d_{+}^{*},d_{-}]
$$ 
acts on $\P(k)^+$ as $\Y_1 \T_{1}\cdots \T_{k-1}$. Therefore, $z_1$ acts on $\P(k)^+$ as $\Y_1$ and, for all $i\geq 1$, the operator $z_i$ acts as $\Y_i$.

The verification of many of the relations in Definition \ref{def: ddpa} is virtually identical to the corresponding verification in \cite{CM}. We briefly indicate the main details.

The fact that \eqref{T relation} and \eqref{T quadratic} hold is clear (also part of Theorem \ref{thm: +standardrep}). For the first relation in \eqref{missing d- relation}, see the proof of the corresponding relation in \cite{CM}*{Lemma 5.3}; 
the second set of relations in \eqref{missing d- relation} is clear from the definition of the maps involved. For  \eqref{d+ relation}, recall the expressions for $\widetilde{\omega}^{-1}_k$ in Remark \ref{rem: otilderels}. With this in mind, the first 
relation in \eqref{d+ relation}  follows from the equality
\begin{equation}\label{eq9}
T_1\cdots T_{k+1} T_1\cdots T_k= T_2\cdots T_{k+1} T_1\cdots T_{k+1},
\end{equation}
which is a consequence of the braid relations. Indeed,
\begin{align*}
\T_1\partial_{k+1}\partial_k&=\T_1 \T_1\cdots \T_{k+1} \T_1\cdots \T_k \X_k \X_{k+1}\\
&=\T_1 \T_2\cdots \T_{k+1} \T_1\cdots \T_{k+1} \X_k \X_{k+1},
\end{align*}
which on $\P(k)^+$ acts as $\T_1 \T_2\cdots \T_{k+1} \T_1\cdots \T_{k} \X_k \X_{k+1}=\partial_{k+1}\partial_k$. The second set of relations in \eqref{d+ relation}  is again a straight consequence of the braid relations.

For the relations \eqref{d- relation}, remark that \eqref{eq8}, \eqref{missing d- relation}, and the fact that $\partial^-_k$ commutes with $\X_1$ implies the first relation in  \eqref{d- relation}. The second relation in  \eqref{d- relation} follows 
from \eqref{eq8} and \eqref{eq9}.

The first relation in \eqref{d+* relation} is a consequence of the fact that an element in the image of  $\partial_{k+1}^*\partial_{k}^{*}$ is symmetric in $x_1, x_2$. The second relation in \eqref{d+* relation} is essentially \eqref{omega2 
relation}.

For the first relation in \eqref{d-2 relation} can be seen (with the help of \eqref{missing d- relation}) to be equivalent to 
\begin{equation}
\Y_1 \partial_k^-=\partial_k^- \Y_1.
\end{equation}
The commutativity relation is not immediately clear from the definition of the operators involved. We proceed as in the proof of \cite{CM}*{Proposition 6.3}.
More precisely, 
\begin{align*}
(\tc+1)(\partial^- [\partial^*,\partial^-]\T_{k-1}-\tc [\partial^*,\partial^-] \partial^-)&=(\tc+1)\partial^-\partial^*\partial^-(\T_{k-1}+\tc) - \tc \partial^* (\partial^-)^2(\T_{k-1}+\tc -\T_{k-1}+1)\\
&+(\partial^-)^2 \partial^* (\T_{k-1}+\tc -\T_{k-1}+1)\\
&= ((\tc+1)\partial^-\partial^*\partial^- - \tc \partial^* (\partial^-)^2 +(\partial^-)^2 \partial^*)(\T_{k-1}+\tc)\\
&= (\partial^- [\partial^*,\partial^-]-\tc [\partial^*,\partial^-] \partial^-)(\T_{k-1}+\tc).
\end{align*}
For the second equality we used the relations \eqref{missing d- relation} and \eqref{d+* relation}. The image of $\T_{k-1}+\tc$ lies on the kernel of $\T_{k-1}-1$. Therefore, it is enough to check that 
$$
\partial^- [\partial^*,\partial^-]=\tc [\partial^*,\partial^-] \partial^-
$$
on $\P(k-1)^+\subset \P(k)^+$. By \eqref{eq7}, $\partial^-$ acts as identity on $\P(k-1)^+\subset \P(k)^+$. After examining the action of both sides of the relation on elements in $\P(k-1)^+\subset \P(k)^+$ of the form \eqref{eq10} we see 
that it is enough to establish the equality for the action on the elements $x_{k-1}^n x_k^m+ x_{k-1}^mx_k^n$, $n,m\geq 0$. The rest of the argument in \cite{CM}*{Proposition 6.3} applies to conclude the verification of the first relation in 
\eqref{d-2 relation}.  

For the second relation in \eqref{d-2 relation} we proceed as follows. By examining the action of both sides of the relation on elements of the form \eqref{eq10} we see that it is enough to establish the equality for the action on the 
elements $x_k^n$, $n\geq 0$. By direct computation,
$$
[\partial^*,\partial^-]\partial^* x_k^n=h_n[(1-\tc)\mathbf{X}_{k+1}+(1-\tc)\qc x_1]-h_n[(1-\tc)\mathbf{X}_{k+1}]
$$
and
$$
\T_1\partial^*[\partial^*,\partial^-]x_k^n=h_n[(1-\tc)\mathbf{X}_{k+1}+(1-\tc)\qc x_1+(1-\tc)\qc x_2]-\T_1 h_n[(1-\tc)\mathbf{X}_{k+1}+(1-\tc)\qc x_1],
$$
from which the claimed equality can be readily verified.

The second relation in \eqref{cross relation}  is precisely \eqref{X-omega2 cross relation}. For the first relation in \eqref{cross relation}, it is enough to verify the case $i=1$, that is 
$$
\widetilde{\omega}^{-1}_{k+1} \Y_1=\Y_2 \widetilde{\omega}^{-1}_{k+1}.
$$
Using the first expression for $\widetilde{\omega}^{-1}_{k+1}$ in Remark \ref{rem: otilderels}, this reduces to \eqref{XY cross relations}.

The relation \eqref{last} is proved as follows
\begin{align*}
\qc \tc^{k+1} \X_1 \omega^{-1}_{k+1}&=  \tc^{k+1} \omega^{-1}_{k+1} \X_{k+1}\\
&=  \tc^{k+1} \varpi_{k+1}  \X_{k+1}\\
&= \Y_1 \T_1\cdots \T_k \X_{k+1}\\
&= \Y_1 \widetilde{\omega}^{-1}_{k+1}.
\end{align*}
Therefore,
$$
\Y_1 \partial_k=-\qc \tc^{k+1} \X_1 \partial^*_{k+1},
$$
as desired.
\end{proof}

%%%%%%%%%%%%%%%%%%%%%%%%%%%%%%%%%%%%%%%%%%%%%%%%%%%%%%%%%
\subsection{} \label{sec: isom}
We conclude with the comparison of the representations of $\Atq$ in Theorem \ref{thm: CMrep} and Theorem \ref{thm: quiverrep}. Let us first define
$$
 \gls{Phi}=( \gls{Phik})_{k \geq 0}: \P_\bullet \to V_{\bullet},
$$
as follows
$$\Phi_k:\P(k)^+\cong \P_k^+\otimes \Sym[\mathbf{X}_k] \to V_k;\quad x_i\mapsto y_i,\quad 1\leq i\leq k;\quad \mathbf{X}_k\mapsto \frac{\Xb}{\tc-1}.$$
We note that each $\Phi_k$ is an $\Rat(\tc,\qc)$-algebra isomorphism.

\begin{thm}\label{thm: isom}
$\Phi_\bullet$ is an isomorphism of $\Atq$-representations.
\end{thm}

\begin{proof}
We clearly have $\Phi_k T_i=T_i \Phi_k$ for all $1\leq i\leq k-1$. We will check that the action of $d_+$ satisfies the following equality 
$$
d_{+,k}\Phi_k=\Phi_{k+1}d_{+,k}.
$$
 Therefore we only need to prove the correspondence for $d_{+}, d_{+}^{*}, d_{-}$.

Let $F=x_1^{t_1}...x_k^{t_k} f[\mathbf{X}_k] \in \P(k)^+$. We have,
\begin{align*}
d_{+,k}\Phi_k F &= d_{+,k}y_1^{t_1}...y_k^{t_k} f[\frac{\Xb}{\tc-1}]\\
	&=-T_1...T_{k}(y_1^{t_1}...y_k^{t_k}y_{k+1}f[\frac{\Xb+(\tc-1)y_{k+1}}{\tc-1}])\\
	&=-T_1...T_{k}(y_1^{t_1}...y_k^{t_k}y_{k+1}f[\frac{\Xb}{\tc-1}+y_{k+1}])\\
	&=-\Phi_{k+1}(T_1...T_{k}x_1^{t_1}...x_k^{t_k}x_{k+1}f[\mathbf{\Xb}_{k+1}+x_{k+1}])\\
	&=-\Phi_{k+1}(T_1...T_{k}X_{k+1}(x_1^{t_1}...x_k^{t_k}f[\mathbf{\Xb}_{k}]))\\
	&=\Phi_{k+1} d_{+,k} F.
\end{align*}
Similar computations show that $d^*_{+,k}\Phi_k=\Phi_{k+1}d^*_{+,k}$ and $d_{-,k}\Phi_k=\Phi_{k-1}d_{-,k}$.
\end{proof}
The following examples illustrate how the actions of  $\Y_1$ and $z_1$ correspond through the isomorphism $\Phi_\bullet$.
\begin{ex}
By definition, $\Phi_k(x_2^2)=y_2^2$ for any $k\geq 2$. From Example \ref{exp: p_2iii} and Example \ref{exp: x_2^2iv} we have
\begin{align*}
\Y_1 \cdot  x_2^2&=\qc^2 \tc(\tc-1)x_1^2+\qc\tc(\tc-1)x_1 x_2-\qc\tc(\tc-1)^2 x_1 e_1[\Xb_2],\\
z_1\cdot  y_2^2 &=\qc^2 \tc(\tc-1)y_1^2+\qc\tc(\tc-1)y_1 y_2-\qc\tc(\tc-1)^2 (y_1y_3+\cdots+y_1y_k)-\qc\tc(\tc-1)y_1 e_1[\Xb].
\end{align*}
Now, from the definition of $\Phi_k$, $k\geq 2$, we can see that
$$
\Phi_k(\Y_1 \cdot x_2^2)=z_1\cdot  y_2^2.
$$
It is important to remark that while the action of $\Y_1$ on $x_2^2\in \P(k)^+$ is independent of $k$, the action of $z_1$ on  $y_2^2\in V_k$ depends on $k$.
\end{ex}

\begin{ex}
By definition, $\Phi_1(p_2[\Xb_1])=p_2[\Xb]/(\tc^2-1)$. From Example \ref{exp: p_2iii} and Example \ref{exp: p_2iv} for $k=1$, we have,
\begin{align*}
\Y_1 \cdot  p_2[\Xb_1]&=-\qc^2 \tc x_1^2+\qc\tc(\tc-1)x_1 e_1[\Xb_1],\\
z_1 \cdot  p_2[\Xb]/(\tc^2-1)&=-\qc^2 \tc  y_1^2+\qc\tc y_1 e_1[\Xb].
\end{align*}
Again,  we can see that
$$
\Phi_1(\Y_1 \cdot p_2[\Xb_1])=z_1\cdot p_2[\Xb]/(\tc^2-1).
$$
For $k=2$, we have, $\Phi_2(p_2[\Xb_1])=y_2^2+p_2[\Xb]/(\tc^2-1)$ and from Examples \ref{exp: x_2^2iv} and \ref{exp: p_2iv} we have,
\begin{align*}
\Y_1 \cdot p_2[\Xb_1]&=-\qc^2 \tc x_1^2+\qc\tc(\tc-1)x_1 e_1[\Xb_1],\\
z_1 \cdot (y_2^2+p_2[\Xb]/(\tc^2-1))&=(\qc^2 \tc(\tc-1)y_1^2+\qc\tc(\tc-1)y_1 y_2-\qc\tc(\tc-1)y_1 e_1[\Xb]) +(-\qc^2 \tc^2y_1^2  + \qc\tc^2y_1 e_1[\Xb])\\
&= -\qc^2 \tc y_1^2 + \qc\tc(\tc-1)y_1 y_2+ \qc\tc y_1 e_1[\Xb].
\end{align*}
From the definition of $\Phi_2$ we can see that
$$
\Phi_2(\Y_1 \cdot p_2[\Xb_1])=z_1\cdot  (y_2^2+p_2[\Xb]/(\tc^2-1)).
$$
Again, the action of $\Y_1$ on $p_2[\Xb_1]\in \P(k)^+$ is independent of $k$, while the action of $z_1$ on  $y_2^2+p_2{\Xb}/(\tc^2-1)\in V_k$ depends on $k$.
\end{ex}

We emphasize that above result is different from the comments in \cite{CM}*{pg. 694} that suggest a possible algebra isomorphism  between  $e_k\Atq e_k$ (the subalgebra of $\Atq$ generated by loops based at node $k$) and a \qq{partially symmetrized}  version of the stable limit \emph{spherical} DAHA, still to be defined for $k>0$. In particular, $\H^+$ does not contain a copy of the stable limit spherical DAHA which, according to the expectation in \cite{CM}, would be isomorphic to $e_0\Atq e_0$.

\subsection*{Competing interests} The authors declare none.

\section*{Index of Notation}
\renewcommand{\glossarysection}[2][]{}
%\begin{multicols}{4}
\printnoidxglossaries
%\end{multicols}

%%%%%%%%%%%%%%%%%%%%%%%%%%%%%%%%%%%%%%%%%%%%%%%%%%%%%%%%%

\end{document}